\documentclass[11pt]{article}
\usepackage{amsmath,amssymb,amsthm}

\newcommand{\ssa}{\ensuremath{_{\alpha}}}
\newcommand{\ssb}{\ensuremath{_{\beta}}}
\newcommand{\ssg}{\ensuremath{_{\gamma}}}

\newcommand{\lla }{\ensuremath{\ell\ssa}}
\newcommand{\llb}{\ensuremath{\ell\ssb}}
\newcommand{\llg}{\ensuremath{\ell\ssg}}

\newcommand{\Tbar}{\ensuremath{\overline{\mathcal{T}}}}

\newcommand{\grad}{\operatorname{grad}}
\newcommand{\red}{\operatorname{red}}

\newcommand{\qedd}{\hfill \ensuremath{\Box}}

\newcommand{\mathC}{\mathbb{C}}
\newcommand{\mathH}{\mathbb{H}}
\newcommand{\mathR}{\mathbb{R}}
\newcommand{\mathZ}{\mathbb{Z}}

\newcommand{\caB}{\mathcal{B}}
\newcommand{\caC}{\mathcal{C}}
\newcommand{\caD}{\mathcal{D}}
\newcommand{\caT}{\mathcal{T}}
\newcommand{\caF}{\mathcal{F}}
\newcommand{\caH}{\mathcal{H}}
\newcommand{\caM}{\mathcal{M}}

\newtheorem{definition}{Definition}
\newtheorem{lemma}[definition]{Lemma}
\newtheorem{theorem}[definition]{Theorem}
\newtheorem{corollary}[definition]{Corollary}

\newtheorem{example}[definition]{Example}

\usepackage{graphicx}
\graphicspath{%
    {converted_graphics/}
    {/}
    {CBMS/CBMSbook/}
}
\begin{document}

\title{Products of twists, geodesic-lengths and Thurston shears}         
\author{Scott A. Wolpert\footnote{Partially supported by National Science Foundation grant DMS - 1005852.}}        
\date{January 4, 2013}          
\maketitle

\begin{abstract}
Thurston introduced shear deformations (cataclysms) on geodesic laminations - deformations including left and right displacements along geodesics. For hyperbolic surfaces with cusps, we consider shear deformations on  disjoint unions of ideal geodesics.  The length of a balanced weighted sum of ideal geodesics is defined and the Weil-Petersson (WP) duality of shears and the defined length is established. The Poisson bracket of a pair of balanced weight systems on a set of disjoint ideal geodesics is given in terms of an elementary $2$-form.  The symplectic geometry of balanced weight systems on ideal geodesics is developed.  Equality of the Fock shear coordinate algebra and the WP Poisson algebra is established.  The formula for the WP Riemannian pairing of shears is also presented.
\end{abstract}  

\section{Introduction}

As a generalization of the Fenchel-Nielsen twist deformation for a simple closed curve, Thurston introduced earthquake deformations for measured geodesic laminations.  Later in his study of minimal stretch maps, Thurston generalized earthquakes to shears (cataclysms),  deformations incorporating left and right displacements \cite{Thurstre}.  Bonahon subsequently developed the fundamental theory of shear deformations in a sequence of papers \cite{Bonshear,BontHd,BonTran}.  At the same time, Penner developed a deformation theory of Riemann surfaces with cusps by considering shear deformations on disjoint ideal geodesics triangulating a surface \cite{Pendec,Pen,Penbk}.  More recently shear deformations play a basic role in the Fock-Goncharov work on the quantization of Teichm\"{u}ller space \cite{Fk,FkChk,FkGn} and in the Kahn-Markovic work on the Weil-Petersson Ehrenpreis conjecture \cite{KMwp}.

The Weil-Petersson (WP) geometry of Teichm\"{u}ller space is recognized as corresponding to the hyperbolic geometry of Riemann surfaces.  For example, twice the dual in the WP K\"{a}hler form of a Fenchel-Nielsen twist deformation is the differential of the associated geodesic-length function.  Also for example, the WP Riemannian pairing of twist deformations is given by a sum of lengths of orthogonal connecting geodesics, see Theorem \ref{gradpr} and \cite{Rier}. An infinitesimal shear on a disjoint union of ideal geodesics is specified by weights on the geodesics with vanishing sum of weights for the edges entering each cusp.  
We define the length of a balanced sum of ideal geodesics and find that twice the dual in the WP K\"{a}hler form of a shear is the differential of the defined length.  We then present the basic WP symplectic and Hamiltonian geometry in Section \ref{sympgeom} with Theorem \ref{twlth} and Corollaries \ref{commshear}, \ref{lpr} and \ref{hh}.  The results include new formulas for the K\"{a}hler form.  We show that the Poisson bracket of a pair of weight systems on a common set of triangulating ideal geodesics is given in terms of an elementary $2$-form computed from the weights alone.  In Section \ref{alge}, we use the elementary $2$-form to show in Theorem \ref{FkWP} that the Fock shear coordinate algebra introduced in the quantization of Teichm\"{u}ller space is the WP Poisson algebra.  The basic WP Riemannian geometry of shears is developed in Section \ref{riemmgeom} with Theorem \ref{shpr}.  We generalize Riera's WP inner product formula and show that the Riemannian pairing of two weight systems on ideal geodesics is given by the combination of an invariant of the geometry of ideal geodesics entering a cusp and a sum of lengths of orthogonal connecting geodesics.  

There are challenges in calculating shear deformations.  In contrast to earthquake deformations, shear deformations are in general not limits of Fenchel-Nielsen twists and a shear on a single geodesic deforms a complete hyperbolic structure to an incomplete structure.  For the deformation theory larger function spaces are involved; for earthquakes geodesic laminations carry transverse Borel measures and for shears geodesic laminations carry transverse H\"{o}lder distributions.  A general approach would require a deformation theory of incomplete hyperbolic structures.  Rather, we follow the approach of \cite{Wlext} and double a surface with cusps across cusps, and open cusps to collars to obtain approximating compact surfaces with reflection symmetries. Shears are then described as limits of opposing twists.  Given the above expectations, the approximating formulas include individual terms that diverge with the approximation.  The object is to show that diverging terms cancel and to calculate the remaining contributions. We use the Chatauby topology for representations to show that the hyperbolic structures converge and an analysis of holomorphic quadratic differentials to show that infinitesimal deformations converge.  

We begin considerations in Section \ref{glfs} with the variation of cross ratio and geodesic-length.  A unified treatment is given for Gardiner's geodesic-length formula \cite{Gardtheta}, Riera's twist Riemannian product formula \cite{Rier} and the original twist-length cosine formula \cite{Wlsymp}.  In Section \ref{Thurshear}, we review Bonahon's results on shears on compactly supported geodesic laminations and Penner's results on shears on ideal geodesics triangulating a surface with cusps.  The review includes the Thurston-Bonahon Theorem that shears on a maximal geodesic lamination are transitive on Teichm\"{u}ller space and Penner's Theorem on $\lambda$ and $h$ length global coordinate. We include the Bonahon-S\"{o}zen and Papadopoulos-Penner results that in appropriate settings the WP K\"{a}hler form is a multiple of the Thurston symplectic form.  In Sections \ref{Thuroppos} and \ref{Chat}, beginning with hyperbolic collars and cusps, we give the geometric description of shear deformations and describe the convergence of opposing twists to shears. In Section \ref{results}, we treat the convergence of infinitesimal opposing twists to infinitesimal shears.  The analysis includes the convergence of holomorphic quadratic differentials.  In Section \ref{sympgeom}, we define the length of a balanced sum of ideal geodesics and establish the basic symplectic geometry results in Theorem \ref{twlth} and the following corollaries.  In Corollary \ref{commshear}, we show that the Poisson bracket of length functions and the shear derivative of a length function are given by evaluation of the elementary $2$-form. We consider the Fock shear coordinate algebra in Section \ref{alge}. We use Penner's topological description of the shear coordinate bracket and compute with the elementary $2$-form to show that the algebra is the WP Poisson algebra.  In Section \ref{riemmgeom} we begin with expansions for gradient pairings for geodesics crossing short geodesics.  Then in Theorem \ref{shpr}, we provide the formula for the WP Riemannian pairing of balanced sums of ideal geodesics.  In Example \ref{Dedekind} we calculate the pairing for the Dedekind $PSL(2;\mathbb Z)$ tessellation to find an exact distance relation.  Finally in Section \ref{circ} we give the length parameter expansion for the sum of lengths of circuits about a closed geodesic.

It is my pleasure to thank Joergen Andersen, Robert Penner, Adam Ross and 
Dragomir \v{S}ari\'c for many helpful conversations and valuable suggestions. 

\section{Gradients of geodesic-lengths}\label{glfs}       

We begin with the basics of deformation theory of Riemann surfaces \cite{Ahqc,Hbbk,ImTan}.  A conformal structure is described by its uniformization.  An infinitesimal variation of a conformal structure is described by a variation of the identity map for the universal cover.  The interesting case for the present considerations is for a Riemann surface of finite type, a compact surface with a finite number of points removed, covered by the upper half plane $\mathbb H$.  For a vector field $v$ on the universal cover and parameter $\epsilon$, there is a variation of the identity map $w_{\epsilon}(z)=z+\epsilon v+o(\epsilon)$, for $z$, respectively $w$, conformal coordinates for the domain and range universal covers.  Provided the vector field is deck transformation group invariant, the map is equivariant with respect to deck transformation groups.  The range conformal structure is described by the angle measure $\arg(dw_{\epsilon})$ for the differential $dw_{\epsilon}=w_{\epsilon,z}dz+w_{\epsilon,\bar z}\overline{dz}$.   The expansion for the variation provides that 
$dw_{\epsilon}=w_{\epsilon,z}(dz\,+\,\epsilon v_{\bar z}\overline{dz})\,+\,o(\epsilon)$, and thus $\arg(dw_{\epsilon}) =\arg(w_{\epsilon,z})\,+\,\arg(dz+\epsilon v_{\bar z}\overline{dz})$.   The derivative of the vector field $v_{\bar z}$ describes the infinitesimal variation of the conformal structure.  The quantity $v_{\bar z}$ is an example of a Beltrami differential, a tensor of type $\frac{\partial}{\partial z}\otimes\overline{dz}$.

For a Riemann surface $R$ of finite type and vector field $v$ defined on the surface (equivalently on the universal cover and invariant by deck transformations), then $w_{\epsilon}(z)$ is a variation of the identity map of the surface and in effect describes a relabeling of the points of the surface - the deformation is trivial.  Nontrivial deformations are given by vector fields on the universal cover; vector fields with nontrivial group cocycles relative to the deck transformation group.  

We consider $B(\mathH)$, the space of Beltrami differentials on $\mathH$, bounded in $L^{\infty}$.  By potential theory considerations, for $\mu\in B(\mathH)$ there is a vector field $v$ on $\mathH$ with $v_{\bar z}=\mu$, that is actually continuous on $\overline \mathH$ and is bounded as $O(|z|\log|z|)$ at infinity \cite{AB}.  In particular elements of $B(\mathH)$ also describe variations of the points of $\mathR$.  We are interested in the corresponding variational formula.

The cross ratio of points of $\mathbb P^1$ is given as
\[
(p,q,r,s)\,=\,\frac{(p-r)(q-s)}{(p-s)(q-r)}
\]
and for $q=s+\Delta s$ and rearranging variables, we obtain a holomorphic $1$-form
\[
\Omega_{pq}(z)\,=\,\frac{(p-q)dz}{(z-p)(z-q)}\,=\,\frac{dz}{(z-p)}\,-\,\frac{dz}{(z-q)}.
\]
The cross ratio and $1$-form are invariant by the diagonal action of $PSL(2;\mathC)$ on all variables. 

There is a natural pairing of Beltrami differentials with $Q(\mathH)$, the space of integrable holomorphic quadratic differentials on $\mathH$,
\[
\quad(\mu,\psi)\,\rightarrow\,\int_{\mathH} \mu\psi \quad\mbox{ for }\mu\in B(\mathH)\mbox{ and } \psi\in Q(\mathH).
\]
Rational functions, holomorphic on $\mathH$, with at least three simple poles on $\mathR$ are example elements of $Q(\mathH)$.  The holomorphic quadratic differentials $Q(\mathH)$ describe cotangents of the deformation space of conformal structures.  The variational formula for points of $\mathR$ is fundamental.
\begin{theorem}\textup{Variation of the cross ratio \cite{Ahsome,Ahqc}.}\label{crvr} For $p,q,r,s\in\mathR$ the variational differential of the cross ratio is
\[
d\log(p,q,r,s)\,=\,-\frac{2}{\pi}\Omega_{pq}\Omega_{rs}\,\in Q(\mathH).
\]
\end{theorem}

The quadratic differentials $Q(\mathH)$ form a pre-inner product space with a densely defined Hermitian pairing
\[
\langle\phi,\psi\rangle\,=\,\int_{\mathH}\phi\bar\psi\,(ds^2)^{-1}\quad\mbox{ for }\phi,\psi\in Q(\mathH)\cap L^2
\]
and $ds^2$ the hyperbolic metric.  The pairing is the Weil-Petersson pre-inner product \cite{Ahsome,Wlcbms}.  The pairing provides formal dual tangent vectors for the differentials of cross ratios
\[
\grad\log(p,q,r,s)\,=\,\overline{(d\log(p,q,r,s))}(ds^2)^{-1}.
\]
We are interested for distinct quadruples $\mathcal P=(p_1,p_2,r_1,r_2),\ \mathcal F=(f_1,f_2,g_1,g_2)$ in the pairing
\[
\langle\grad\log\mathcal P,\grad\log\mathcal F\rangle.
\]
The pairing is continuous in the quadruples for all points distinct and also is continuous for $(r_1,r_2)$ tending to $(p_1,p_2)$ and $(g_1,g_2)$ tending to $(f_1,f_2)$.  We will evaluate  particular configurations for the pairing. 

Let $\caT$ be the Teichm\"{u}ller space of homotopy marked genus $g$, $n$ punctured Riemann surfaces $R$ of negative Euler characteristic.  We are interested in pairings corresponding to geometric constructions of deformations.  A point of $\caT$ is the equivalence class of a pair $(R,f)$ with $f$ a homeomorphism from a reference topological surface $F$ to $R$.  By the Uniformization Theorem a conformal structure determines a unique complete compatible hyperbolic metric $ds^2$ for $R$ and a deck transformation group $\Gamma\subset PSL(2;\mathbb R)$ with $R=\mathH/\Gamma$.  The Teichm\"{u}ller space is a complex manifold with cotangent space at $R$ represented by $Q(R)$, the space of holomorphic quadratic differentials on $R$ with at most simple poles at punctures.

The pairing 
\[
\quad(\mu,\psi)\,\rightarrow\,\int_R\mu\psi\quad\mbox{for } \mu\in B(R)\mbox{ and } \psi\in Q(R)
\]
is the ingredient for Serre duality and consequently the tangent space of $\caT$ at $R$ is $B(R)/Q(R)^{\perp}$ \cite{Ahsome,Ahqc,Cmbbk,Hbbk,ImTan}.  The $L^2$ Hermitian pairing
\[
\langle\phi,\psi\rangle\,=\,\int_R\phi\bar\psi\,(ds^2)^{-1}
\]
is the Weil-Petersson (WP) cometric for $Q(R)$.  The metric dual mapping
\[
\phi\,\rightarrow\,\bar\phi(ds^2)^{-1}\,\in Q(R)
\]
is a complex anti linear isomorphism, since Beltrami differentials of the given form (harmonic differentials) give a direct summand of $Q(R)^{\perp}$ in $B(R)$.  The metric dual mapping associates a tangent vector to a cotangent vector and so defines the WP K\"{a}hler metric on the tangent spaces of $\caT$; the mapping is the Hermitian metric gradient.  

Geodesic-lengths and Fenchel-Nielsen twist deformations are geometric quantities for pairings.  Associated to a nontrivial, non peripheral free homotopy class $\alpha$ on the reference surface $F$ is the length $\lla(R)$ of the unique geodesic in the free homotopy class for $R$.  Geodesic-length is given as $2\cosh\lla/2\,=\,\operatorname{tr}A$ for $\alpha$ corresponding to the conjugacy class of $A\in\Gamma$ in the deck transformation group.   Geodesic-lengths are functions on Teichm\"{u}ller space with a direct relationship to WP geometry.  A Fenchel-Nielsen twist deformation is also associated to a closed simple geodesic.  The deformation is given by cutting the surface along the geodesic $\alpha$ to form two metric circle boundaries, which then are identified by a relative rotation to form a new hyperbolic surface.  A flow on $\caT$ is defined by considering the family of surfaces $\{R_t\}$ for which at time $t$ reference points from sides of the original geodesic are relatively displaced by $t$ units to the right on the deformed surface.  The infinitesimal generator the Fenchel-Nielsen vector field $t_{\alpha}$, the differential of the geodesic-length and the gradient of geodesic-length satisfy duality relations
\begin{equation}\label{wpdual}
2\omega_{WP}(\ ,t_{\alpha})\,=\,d\lla\quad\mbox{and equivalently}\quad 2t_{\alpha}\,=\,J\grad\lla,
\end{equation}
for $\omega_{WP}$ the WP K\"{a}hler form and $J$ the complex structure of $\caT$ (multiplication by $i$ on $B(R)/Q(Q)^{\perp}$) \cite{WlFN,Wlcbms}.  The factor of $2$ adjustment to our formulas as detailed in \cite[\S 5]{Wlcusps} is included. 

We are interested in the WP metric and Lie pairings of the infinitesimal deformations $\grad\lla$ and $t_{\alpha}$ with geodesic-length functions $\llb$.  The formulas begin with Gardiner's calculation of the differential of geodesic-length.  We now use a single simplified approach that provides Gardiner's $d\lla$ formula \cite{Gardtheta}, the cosine formula for $t_{\alpha}\llb$ \cite{Wlsymp,Wlcbms}, the sine-length formula for $t_{\alpha}t_{\beta}\llg$ \cite{Wlsymp,Wlcbms}, as well as Riera's length-length formula for $\langle\grad\lla,\grad\llb\rangle$ \cite{Rier,Wlcbms}.  The approach combines Theorem \ref{crvr}, coset decompositions for the uniformization group and calculus calculations.  An important step is identifying a telescoping sum corresponding to a cyclic group action.  We present the approach.

\begin{theorem}\textup{Gardiner's variational formula \cite{Gardtheta}.}\label{Gardtheta} For a closed geodesic $\alpha$,
\[
d\lla\,=\,\frac{2}{\pi}\sum_{C\in\langle A\rangle\backslash\Gamma}\Omega_{r_A a_A}^2(Cz)\, \in Q(R)
\]
with $\alpha$ corresponding to the conjugacy class of $A\in\Gamma$ with repelling fixed point $r_A$ and attracting fixed point $a_A$.  
\end{theorem} 
\begin{proof}
We begin with the geodesic-length.  For a hyperbolic transformation $A$, the geodesic-length is $\log(As,s,r_A,a_A)$ for $s$ a point of $\mathR$ distinct from the fixed points.  We begin with the variational formula for the cross ratio from Theorem \ref{crvr}.  The resulting integrand is in $L^1(\mathH)$ and $\mathH$ is the disjoint union 
\[
\bigcup_{n\in\mathZ}\,\bigcup_{C\in\langle A\rangle\backslash\Gamma}A^nC(\mathcal F)
\]
for $\mathcal F$ a $\Gamma$ fundamental domain. By a change of variables the union over domains is replaced by a sum of integrands
\begin{equation}\label{unfolded}
d\lla[\mu]\,=\,-\frac{2}{\pi}\Re\int_{\mathcal F}\mu\sum_n\sum_{C\in\langle A\rangle\backslash\Gamma}\Omega_{As\,s}(A^nCz)\Omega_{r_Aa_A}(A^nCz).
\end{equation}
The invariance of $\Omega$ by the diagonal $PSL(2;\mathR)$ action gives $\Omega_{pq}(A^nw)\,=\,\Omega_{A^{-n}pA^{-n}q}(w)$ and the given product of forms is
\[
\Omega_{A^{-n+1}sA^{-n}s}(Cz)\,\Omega_{r_Aa_A}(Cz).
\]
Using the $\Omega$ partial fraction expansion, the first factor is
\[
\Omega_{A^{-n+1}sA^{-n}s}\,=\,\frac{dw}{(w-A^{-n+1}s)}\,-\,\frac{dw}{(w-A^{-n}s)}
\]
and the integer sum telescopes
\[
\sum_{n=-N}^N\Omega_{A^{-n+1}sA^{-n}s}\,=\,\Omega_{A^{N+1}sA^{-N}s}
\]
and as $N$ tends to infinity, $A^{N+1}s$ tends to $a_A$ and $A^{-N}s$ tends to $r_A$.  (Various forms of the telescoping appear in the calculations for the cosine formula \cite[pgs. 220-221]{Wlsymp}, the sine-length formula \cite[pgs. 223-224]{Wlsymp} and the length-length formula \cite[pgs. 113-114]{Rier}.)  The sum in (\ref{unfolded}) now becomes the desired sum 
\[
-\sum_{C\in\langle A\rangle\backslash\Gamma}\Omega_{r_A a_A}^2(Cz).
\]
\end{proof}

We consider the WP Hermitian pairing of gradients $\langle\grad\lla,\grad\llb\rangle$.  By (\ref{wpdual}) the imaginary part of the pairing is
\[
\Re\langle J\grad\lla,\grad\llb\rangle\,=\,2t_{\alpha}\llb\,=\,2\sum_{p\in\alpha\cap\beta}\cos\theta_p.
\]
The real part of the pairing $\langle\grad\lla,\grad\llb\rangle$ was first evaluated by Riera \cite{Rier}.  We now apply the above approach and with a single simpler treatment derive the real and imaginary part formulas.  Riera's formula involves the logarithmic function
\[
R(u)\,=\,u\log\Big|\frac{u+1}{u-1}\Big|\,-\,2.
\]
The function is even with a logarithmic singularity at $\pm 1$ and with the expansion
\[
R(u)\,=\,2\,(\frac{1}{3u^2}\,+\,\frac{1}{5u^4}\,+\,\frac{1}{7u^6}\,+\,\cdots)\quad\mbox{for }|u|>1.
\]
In particular for $u>1$, the function and its even derivatives are positive and 
the function is $O(u^{-2})$ for $u>1$.  The function $R(u)$ is also given as
\[
\frac{u}{2}\tanh^{-1} \frac{1}{u}\,-\,2\quad\mbox{for }|u|>1\quad\mbox{and }\quad
\frac{u}{2}\tanh^{-1} u\,-\,2\quad\mbox{for }|u|<1.
\]
We present the pairing formula for the general case of a cofinite group possibly with parabolic and elliptic elements.  

\begin{theorem}\textup{The complex gradient pairing \cite{Wlsymp,Rier}.}\label{gradpr} For closed primitive geodesics $\alpha,\beta$ corresponding to elements $A,B\in\Gamma$, we have for the WP pairing
\[
\langle\grad\lla,\grad\llb\rangle\,=\,\frac{2}{\pi} \delta_{\alpha\beta}e(A)\lla\,+\,\sum_{D\in\langle A\rangle\backslash\Gamma\slash\langle B\rangle} \mathcal R_D,
\]
where $\delta_{\alpha\beta}$ is the Kronecker delta for the geodesic pair, where $e(A)$ is $2$ in the special case of the axis of $A$ having order-two elliptic fixed points and is $1$ otherwise, 
where for the axes $\operatorname{axis}(A),\operatorname{axis}(DBD^{-1})$ disjoint in $\mathH$, then
\[
\mathcal R_D\,=\,\frac{2}{\pi}R(\cosh d(\operatorname{axis}(A),\operatorname{axis}(DBD^{-1})))
\]
and for the axes intersecting with angle $\theta_D$, then
\[
\mathcal R_D\,=\,\frac{2}{\pi}R(\cos\theta_D)\,-\,2i\cos\theta_D.
\]
 Twist-length duality and $J$ an isometry provide that $4\langle t_{\alpha},t_{\beta}\rangle\,=\,\langle\grad\lla,\grad\llb\rangle$.  
\end{theorem}
\begin{proof}
For A a hyperbolic element we write
\[
\Theta_A\,=\,\sum_{C\in\langle A\rangle\backslash\Gamma}\Omega_{r_Aa_A}^2
\]
and from Gardiner's formula $d\lla=(2/\pi)\Theta_A$ with
\[
\langle\Theta_A,\Theta_B\rangle\,=\,-\int_{\mathH}\,\sum_{C\in\langle A\rangle\backslash\Gamma}\Omega_{r_Aa_A}^2(Cz)\,\overline{\Omega_{Bs\,s}(z)}\,\overline{\Omega_{r_Ba_B}(z)}\,(ds^2)^{-1}.
\]
We first decompose each left coset $\langle A\rangle\backslash\Gamma$ by considering right $\langle B\rangle$ cosets and then move the $\langle B\rangle$ action to the two conjugate forms.  The resulting sum over $\langle B\rangle$ is telescoping.  In particular, we enumerate the cosets of the sum by writing for $C\in\langle A\rangle\backslash\Gamma$ the decomposition $C=DB^n,\,D\in\langle A\rangle\backslash\Gamma\slash\langle B\rangle$ for $n\in\mathbb Z$.  For $A,B$ primitive hyperbolic elements, we consider uniqueness of the presentation of an element of $\langle A\rangle D$ in the form $A^mDB^n$. A non unique presentation is equivalent to a solution of $A^a=DB^bD^{-1}$  for a non trivial integer pair $(a,b)$.  Since $A,B$ each generate $\Gamma$ maximal cyclic subgroups, a non trivial solution of $A^a=DB^bD^{-1}$ provides that $A$ is conjugate to $B^{\pm1}$ by the element $D$.  In particular the presentation $A^mDB^n$ is unique except for the case $\alpha=\beta$ with $A=DB^{\pm1}D^{-1}$.  In the case $\alpha=\beta$ we select the element $A$ to represent the geodesic and the presentation is unique except for the case of $D$ either the identity or the special case of $\Gamma$ containing an order-two elliptic $E$ with $A=EA^{-1}E$.  For the special cases there is no distinction between left and right 
$\langle A\rangle$ cosets; we only use left cosets.  The special left cosets are for the identity element and the element $E$.  

Now for each resulting integral of the sum, change variable by writing $w=B^nz$; the effect is to move a $B^{-n}$ action to the variable of $\Omega_{Bs\,s}\Omega_{r_Ba_B}$.  Using the diagonal $PSL(2;\mathR)$ invariance of $\Omega$, the $B^{-n}$ action is moved to the quadruple of points, resulting in the telescoping sum
\[
\sum_{n\in\mathbb Z}\,\Omega_{B^{n+1}sB^ns}\Omega_{r_Ba_B}\,=\,-\Omega_{r_Ba_B}^2.
\]
The result is the general formula
\begin{multline}\label{genform}
\langle\Theta_A,\Theta_B\rangle\,=\,-\delta_{AB^{\pm 1}}e(A)\int_{\mathH}\Omega_{r_Ba_B}^2\,\overline{\Omega_{Bs\,s}}\,\overline{\Omega_{r_Ba_B}}\,(ds^2)^{-1}\,+\\ \sum_{D\in\langle A\rangle\backslash\Gamma\slash\langle B\rangle}\,\int_{\mathH}\Omega_{r_{D^{-1}AD}a_{D^{-1}AD}}^2\,\overline{\Omega_{r_Ba_B}^2}\,(ds^2)^{-1},
\end{multline}
where the Kronecker delta indicates that the first integral is only present for the case that $A=B^{\pm 1}$, $e(A)$ is $2$ in the case of order-two elliptic fixed points on the axis of $A$ and is otherwise $1$, and for the second integral the diagonal invariance was used to move the $D$ action to the pair of points.  For each integral, a change of variable by an element of $PSL(2;\mathR)$ results in the inverse element applied to the tuple of points.  It follows that the first integral depends only on the $PSL(2;\mathR)$ conjugacy class of $B$ and the second integral depends only on the $PSL(2;\mathR)$  class of the pair $(D^{-1}AD,B)$.  It follows that the first integral is a function of the geodesic-length for $B$ and the second integral depends only on the distance between/intersection angle of the axes. 

We evaluate the integrals.  The differential $\Omega_{pq}$ is continuous in $p,q$, including at infinity; for $q$ tending to infinity the form limits to $dz/(z-p)$.   For the first integral of (\ref{genform}), we take the pair of points to be $0$ and $\infty$, to obtain for $z=re^{i\theta}$ the integral
\[
-\int_{\mathH}\,\frac{1}{z^2\bar z}\,\overline{\frac{(Bs-s)}{(z-Bs)(z-s)}}\,r^2\sin^2\theta\,rdrd\theta,
\]
which for $P=e^{i\theta}Bs,\,Q=e^{i\theta}s$ becomes
\begin{multline*}
-\int_0^{\pi}\int_0^{\infty}\frac{(P-Q)}{(r-P)(r-Q)}\sin^2\theta\, drd\theta\,=\\ 
\log\frac{(r-Q)}{(r-P)}\bigg|_0^{\infty}\int_0^{\pi}\sin^2\theta\, d\theta\,=\,\frac{\pi}{2}\log\frac{Bs}{s},
\end{multline*}
as expected, since $\grad\ell_*=2/\pi\,\Theta_*$.   For the second integral of (\ref{genform}), we take the first pair of points to be $0$ and $\infty$, to obtain the integral
\[
\int_{\mathH}\,\frac{1}{z^2}\,\overline{\bigg(\frac{(p-q)}{(z-p)(z-q)}\bigg)^2}\,r^2\sin^2\theta\, rdrd\theta,
\]
which for $P=e^{i\theta}p,\,Q=e^{i\theta}q,$ becomes
\begin{equation}\label{mainint}
\int_0^{\pi}\int_0^{\infty}\frac{(P-Q)^2}{(r-P)^2(r-Q)^2}\,\sin^2\theta\,rdrd \theta.
\end{equation}
The $r$ integral has antiderivative
\[
\frac{(P+Q)}{(P-Q)}\log\frac{(r-Q)}{(r-P)}\,-\,\frac{P}{(r-P)}\,-\,\frac{Q}{(r-Q)}.
\]
We are evaluating an area integral and $\theta$ varies in the interval $(0,\pi)$; for $p,q\in\mathR$, $\theta$ as described, and $r$ real positive, the quotient $(r-Q)/(r-P)$ is valued in the complex open lower half plane.  The antiderivative is invariant under interchanging $p,q$; we now normalize $p$ to be positive real.  We use the principal branch of the logarithm; for $r$ close to zero the argument is close to $-\pi$.  Evaluating $r$ at $0,\infty$ and integrating in $\theta$ gives
\[
\frac{\pi}{2}\big(\frac{\kappa+1}{\kappa-1}\log\kappa\,-\,2\big)\quad\mbox{for }\kappa\ \mbox{the ratio }q/p\,=\,(q,p,0,\infty).
\]
To interpret geometrically, compare to \cite[pg. 114]{Rier}, set $u=(\kappa+1)/(\kappa-1)=2(\infty,q,p,0)-1$, to obtain the complex-valued expression
\[
\frac{\pi}{2}\,\big(u\log\frac{u+1}{u-1}\,-\,2\big).
\]
For the lines $\stackrel{\frown}{0\infty}$ and $\stackrel{\frown}{pq}$ disjoint, the ratio $\kappa=p/q$ is positive and the logarithm is real, with $u=\cosh \delta_*$, for $\delta_*$ the distance between the lines.  For the lines intersecting, the ratio $\kappa=q/p$ is negative and the argument of the logarithm is $-\pi$ and evaluation gives
\[
\frac{\pi}{2}\,R(\cos\theta_*)\,-\,\frac{\pi}{2}i\pi\cos\theta_*,
\]
as desired.

\end{proof}

The double coset enumeration admits a topological/geometric description.  We consider that $\alpha$ and $\beta$ are primitive and $\Gamma$ is torsion-free.  On the surface $R$, consider the homotopy classes rel the closed sets $\alpha, \beta$ of arcs connecting $\alpha$ to $\beta$.  For the universal cover, fix a lifting of $\alpha$ to a line $\tilde\alpha_0$ in $\mathH$; then a connecting homotopy class on $R$  lifts to a homotopy class of arcs connecting $\tilde\alpha_0$ to $\tilde\beta$ (a line lifting of $\beta$).  The relation rel $\alpha$ corresponds to the relation of the $\langle A\rangle$ action on homotopy lifts. In particular, the non trivial classes on $R$ rel $\alpha,\beta$ {\em biject} to the classes in $\mathH$ rel $\tilde\alpha_0, \tilde\beta$, for $\tilde\beta$ (disjoint from $\tilde\alpha_0$) ranging over the line liftings of $\beta$ modulo the action of $\langle A\rangle$; the non trivial classes on $R$ correspond to lines 
$\tilde\beta$ disjoint from $\tilde\alpha_0$.  
To enumerate the pairs $(\tilde\alpha_0,\tilde\beta)$ for $\tilde\beta$ distinct modulo the $\langle A\rangle$ action, for $A$ generating the stabilizer of $\tilde\alpha_0$ and $B$ generating the stabilizer of a line lifting of $\beta$, then line pairs distinct modulo the $\langle A\rangle$ action {\em correspond} bijectively to double cosets by the rule
\[
(\tilde\alpha_0,\tilde\beta)\,=\,(\operatorname{axis}(A),\operatorname{axis}(DBD^{-1}))\quad\mbox{corresponds to}\quad D\in\langle A\rangle\backslash\Gamma\slash\langle B\rangle.
\]
The relation $\operatorname{axis}(DBD^{-1})=D(\operatorname{axis}(B))$ is part of the correspondence.  For a finite number of double cosets the corresponding axes intersect.  Overall the axes enumeration by double cosets, enumerates pairs of line liftings of $\alpha$ and $\beta$ modulo the diagonal action of the group $\Gamma$.  The geometric description comes from the description of a pair of lines.  A pair of lines either intersects or has a unique perpendicular geodesic, minimizing the connecting distance.  The cosine and hyperbolic cosine describe the geometry of the configurations. 

The present approach to evaluating the pairing is a combination and simplification of earlier works.  The role of the cyclic group in Gardiner's formula was first noted by Hejhal \cite[Theorem 4]{Hejmono}.  The telescoping of the cyclic group sums appears in the proofs of Theorem 3.3 and 3.4 of \cite{Wlsymp} and in Theorem 2 of \cite{Rier}, although in each case the telescoping is presented as a special feature.  The basic integral (\ref{mainint}) is simpler than found in the earlier formulations.   The present approach can be applied to evaluate the second twist Lie derivatives $t_{\alpha}t_{\beta}\llg$.  The first derivative $t_{\alpha}\llb$ is a sum of cosines of intersection angles. A cosine is given by a cross ratio, the starting point for the above considerations.  

\section{Thurston shears}\label{Thurshear}

We are interested in Thurston shears (cataclysms) on ideal geodesics for a Riemann surface with cusps.  Thurston studied the shear deformation for compact geodesic laminations \cite{Thurstre}.  Bonahon developed the fundamental results in a sequence of papers \cite{Bonshear,BontHd,BonTran}.  We present a brief summary of Bonahon's basic results following \cite{Bonshear}.  In a series of works \cite{Pendec,Pen,Penbk}, Penner developed a deformation theory of Riemann surfaces with cusps by considering shear deformations on ideal geodesics triangulating a surface.  Our interests include Penner's $\lambda$-length formulas and formulas for the WP K\"{a}hler/symplectic form \cite{PapPen}.  We present a brief summary of Penner's results following the exposition of the book \cite{Penbk}.

A {\em geodesic lamination} $\lambda$ is a closed union of disjoint simple geodesics.  A geodesic lamination for a compact surface $R$ is maximal provided $R-\lambda$ is a union of ideal triangles.   A {\em transverse measure} for a geodesic lamination $\lambda$ is the assignment for each transverse arc $k$ with endpoints in $\lambda^c$ of a positive Borel measure $\mu$ on the transverse arc with $\operatorname{supp}(\mu)=\lambda\cap k$. If transverse arcs $k,k'$ are homotopic through arcs with endpoints in $\lambda^c$ then the assigned measures correspond by the homotopy.   The assignment $k\mapsto \mu(k)$ is additive under countable subdivision of transverse arcs.  A measured geodesic lamination defines an earthquake deformation by interpreting $\mu(k)$ as the relative left shift of the $\lambda$ complementary regions containing the $k$ endpoints.  By allowing left and right shifts on complementary regions, Thurston defined the shear deformation.  The relative left shift of $\lambda$ complementary regions again defines a functional on transverse arcs.  The functional, called a {\em transverse cocycle}, is only finitely additive under subdivision of transverse arcs. A transverse cocycle is not given by integrating a measure, rather is given by elements of the dual of H\"{o}lder continuous functions on transverse arcs.   The space of transverse cocycles $\mathcal H(\lambda)$ on a geodesic lamination is a finite dimensional vector space.  

Teichm\"{u}ller space is the space of isotopy classes of hyperbolic metrics.  A geodesic lamination is represented on each isotopy class of a hyperbolic metric.  Shear deformations on a given maximal geodesic lamination parameterize Teichm\"{u}ller space.  A projection between leaves is defined for the lift of a lamination to the universal covering of the surface.   The construction begins with the observation that the unit area horoballs in an ideal triangle are foliated by horocycles.  The tangent field of the partial foliation of ideal triangles extends to a Lipschitz vector field on the universal covering; the vector field is not defined on the small trilateral regions in each ideal triangle.  The Lipschitz vector field defines a projection between leaves of the lift of the lamination.  The projection defines a relative displacement between lamination complementary regions.  The relative displacement is finitely additive.  The relative left displacement is called the {\em shearing cocycle} $\sigma_R$ of the surface $R$.  The transverse cocycle for the shear deformation from a surface $R_1$ to a surface $R_2$ is the difference $\sigma_{R_1}-\sigma_{R_2}$ of shearing cocycles.  For a train track carrying a geodesic lamination, transverse measures are specified in terms of non negative weights on the track and transverse cocycles are specified in terms of real weights.  We also refer to the Thurston symplectic intersection form $\tau$ for a train track.  The shearing cocycles for a maximal geodesic lamination provide an embedding of Teichm\"{u}ller space.  

\begin{theorem} \textup{\cite[Theorems A, B]{Bonshear}}\label{thrmAB}. The map $R\mapsto \sigma_R$ defines a real analytic homeomorphism from $\caT$ to an open convex cone $\caC(\lambda)$ bounded by finitely many faces in $\caH(\lambda)$.   A transverse cocycle $\mu$ is in the cone $\caC(\lambda)$ if and only if $\tau(\mu,\nu)>0$ for every transverse measure $\nu$ for $\lambda$.
\end{theorem}

The $R$-length $\ell_{\mu}(R)$ of the transverse cocycle $\mu$ for $\lambda$ is a generalization of the total-length of a transverse measure.  The $R$-length is defined as
\[
\ell_{\mu}(R)\,=\,\int\int_{\lambda}d\ell\, d\mu,
\]
computed locally by first integrating hyperbolic length measure along the leaves of $\lambda$ and then integrating the local function on the local space of $\lambda$ leaves with respect to the H\"{o}lder distribution $\mu$.  The $R$-length generalizes the weighted length for weighted simple closed geodesics; $R$-length is given by the Thurston intersection form and the shearing cocycle as follows.
\begin{theorem} \textup{\cite[Theorem E]{Bonshear}}\label{thrmE}. If $\mu$ is a transverse cocycle for the maximal geodesic lamination $\lambda$ and $\sigma_R\in\caH(\lambda)$ is the shearing cocycle of the hyperbolic surface $R$ then $\ell_{\mu}(R)=\tau(\mu,\sigma_R)$.
\end{theorem}

The Theorem \ref{thrmAB} embedding of $\caT$ into the vector space $\caH(\lambda)$ provides identifications of tangent spaces $\mathbf T\caT$ with $\caH(\lambda)$.  The identification enables a comparison of symplectic forms.

\begin{theorem} \textup{\cite{BonSoz}}.  Let $R$ be a compact hyperbolic surface with a maximal geodesic lamination $\lambda$.  Then for the tangent space identifications $\mathbf T\caT \simeq \caH(\lambda)$, the WP K\"{a}hler form is a constant multiple of the Thurston intersection form.
\end{theorem}

A {\em decoration} for a hyperbolic metric with cusps is the designation of a horocycle at each cusp.  Decorated Teichm\"{u}ller space $\caD\caT$ is the space of isotopy classes of hyperbolic metrics with cusps and decorations \cite{Penbk}.  The decorated Teichm\"{u}ller space is naturally fibered over Teichm\"{u}ller space with fibers given by varying the horocycle lengths in a decoration.  A section of the fibration is given by prescribing horocycle lengths.  A decoration enables a notion of relative length for ideal geodesics. The $\lambda$-{\em length} of an ideal geodesic $\alpha$ is $\lambda(\alpha)=e^{\delta(\alpha)/2}$, where $\delta(\alpha)$ is the signed distance along $\alpha$ between the decoration horocycles; the distance is positive in the case that the associated horodiscs are disjoint.  We are interested in the $\lambda$-lengths for the isotopy class of a given ideal triangulation $\Delta$ of hyperbolic metrics.  An ideal triangulation for a genus $g$ surface with $n$ cusps has $6g-6+3n$ ideal geodesics and $4g-4+2n$ triangles.

\begin{figure}[thb] 
  \centering
  \includegraphics[bb=0 0 627 639,width=3in,height=3.06in,keepaspectratio]{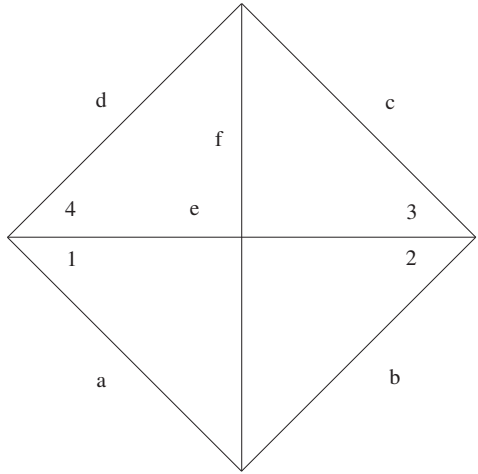}
  \caption{Adjacent ideal triangles with a second diagonal. }
  \label{fig:diamond}
\end{figure}

Additional parameters are associated to an ideal triangulation. The ideal geodesics divide the decoration horocycles into segments.  The $h$-{\em lengths} are the lengths of the horocycle segments.   For an ideal triangle, the lengths are related by $h_{\hat a}=\lambda_a/\lambda_b\lambda_c$, where for the horocycle segment $\hat a$ the triangle opposite side is $a$ and the triangle adjacent sides are $b,c$.  We are particularly interested in the shear coordinates. An ideal triangle has a median.  For a pair of triangles adjacent along an ideal geodesic $\alpha$, drop perpendiculars from the medians to $\alpha$.  The {\em shear coordinate} for $\alpha$ is the signed distance between the median projections; the distance is positive if the projections lie to the right of one another along $\alpha$.  The shear coordinate is given simply in terms of $\lambda$-lengths and $h$-lengths.  In Figure \ref{fig:diamond}, the shear coordinate for the diagonal $e$ is given as 
\begin{equation}\label{shear}
\sigma_e\,=\,\log \frac{\lambda_b\lambda_d}{\lambda_a\lambda_c}\,=\,\log \frac{h_1}{h_4}\,=\,\log \frac{h_3}{h_2},\ 
\textup{\cite[Chap. 1, Corollary 4.16]{Penbk}.} 
\end{equation} 

\noindent The fibers of the Teichm\"{u}ller fibration $\caD\caT\rightarrow \caT$ are characterized simply by constant shear coordinates.  

By the classical result of Whitehead, triangulations with common vertices can be related by a sequence of replacing diagonals in quadrilaterals \cite[Chap. 2, Lemma 1.4]{Penbk}.  The effect on $\lambda$-lengths of replacing diagonals is given by Penner's basic Ptolemy equation $\lambda_{13}\lambda_{24}=\lambda_{12}\lambda_{34}+\lambda_{14}\lambda_{23}$ for the configuration of Figure \ref{fig:diamond}, \cite[Chap. 1, Corollary 4.6]{Penbk}.  We also note the {\em coupling equation} $h_1h_2=h_3h_4$ for the configuration of Figure
 \ref{fig:diamond}; the equation follows from the definition of $h$-lengths.  The $\lambda$ and $h$ lengths provide global coordinates for 
$\caD\caT$.

\begin{theorem}\label{Penthrm}\textup{\cite[Chap. 2, Theorems 2.5, 2.10; Chap. 4, Theorems 2.6, 4.2]{Penbk}.}  For the ideal triangulation $\Delta$, the $\lambda$-length mapping $\caD\caT\rightarrow \mathbb R_{>0}^{\Delta}$ is a real-analytic homeomorphism.   For $V$ the vertex sectors of the ideal triangulation, the $h$-length mapping $\caD\caT\rightarrow \mathbb R_{>0}^V$ is a real-analytic embedding into a real-algebraic quadric variety given by coupling equations.  For the ideal triangulation, the shear coordinate mapping $\caT\rightarrow \mathbb R^{\Delta}$ is a real-analytic homeomorphism onto the linear subspace given by vanishing of the sum of shears around each cusp.   The action of the mapping class group $\operatorname{MCG}$ is described by permutations followed by finite compositions of Ptolemy transformations.  
\end{theorem}  

The WP K\"{a}hler form pulls back to the decorated Teichm\"{u}ller space and has a universal expression in terms of $\lambda$ and $h$ lengths.  We present new formulas for the pullback in Section \ref{results}.
\begin{theorem}\textup{\cite[Chap. 2, Theorem 3.1]{Penbk}.}\label{Penlambda}  For an ideal triangulation $\Delta$, the pullback WP K\"{a}hler form on $\caD\caT$ is
\[
\widetilde{\omega_{WP}}\,=\,\sum_{\Delta}
\widetilde{\lambda_a}\wedge\widetilde{\lambda_{b\,}}+\widetilde{\lambda_{b\,}}\wedge\widetilde{\lambda_{c\,}}+\widetilde{\lambda_{c\,}}\wedge\widetilde{\lambda_{a}},
\]
where the sum is over ideal triangles, $\widetilde{\lambda_*}=d\log\lambda_*$ and the individual triangles have sides $a,b$ and $c$ in clockwise order.
\end{theorem}
\noindent The formula is given without Penner's initial $2$ factor following the adjustment to our own formulas as detailed in \cite[$\S$5]{Wlcusps}.  

Papadopoulos-Penner establish a formula for the pullback $\widetilde{\omega_{WP}}$ in terms of $h$-lengths and describe identifications of spaces to establish that $2\widetilde{\omega_{WP}}$ coincides with Thurston's intersection form \cite{PapPen}.  Specifically the authors show that their change of variable $(\dag\dag)$ transforms their formula $(\dag)$ to the formula $(\dag\dag\dag)$; the calculation applies to the present setting by taking $\mu(greek\ index)=-\log h_{\widehat{index}}$ and $\mu(index)=\log \lambda_{index}$ and noting the factor of $2$.  
\begin{corollary}\textup{\cite{PapPen}.}\label{hform1}  For an ideal triangulation $\Delta$, the pullback WP K\"{a}hler form is
\[
\widetilde{\omega_{WP}}\,=\,\sum_{\Delta}
\widetilde{h_{\alpha}}\wedge\widetilde{h_{\beta}}+\widetilde{h_{\beta}}\wedge\widetilde{h_{\gamma}}+\widetilde{h_{\gamma}}\wedge\widetilde{h_{\alpha}},
\]
where the sum is over ideal triangles, $\widetilde{h_*}=d\log h_*$ and the individual triangles have vertex sectors $\alpha,\beta$ and $\gamma$ in clockwise order.
\end{corollary}
\noindent In particular the $\lambda$ to $h$ change of coordinates is pre symplectic.

Papadopoulos and Penner introduce the formal Poincar\'{e} dual of an ideal triangulation.  The formal dual is a trivalent graph with an orientation for the edges at a vertex.  A modification of the trivalent graph is a punctured null gon train track.  A set of logarithms of $\lambda$-lengths corresponds to a measure on the train track.  A modification of the construction of a measured foliation from a measured train track parameterizes the space $\caD\caM\caF$ of decorated measured foliations.  

\begin{theorem} \textup{\cite[Proposition 4.1]{PapPen}}.  The train track parameterization provides a homeomorphism of $\caD\caT$ to $\caD\caM\caF$.  The homeomorphism identifies twice the pullback WP K\"{a}hler form and the Thurston intersection form  $2\,\widetilde{\omega_{WP}}=\tau$.
\end{theorem}

\section{Thurston shears as limits of opposing twists}\label{Thuroppos}

We show that weighted Fenchel-Nielsen twists with twist lines orthogonal to short geodesics converge to a Thurston shear deformation on ideal geodesics, as the short lengths tend to zero.  We begin with the collars and cusp description \cite{Busbook}.   
For a closed geodesic $\alpha$ on the surface $R$ of length $\lla$, normalize the universal covering for the corresponding deck transformation to be $z\rightarrow e^{\lla}z$.   The collar $\mathcal C(\alpha)=\{\lla/2\le\arg z\le \pi -\lla/2\}/\langle z\rightarrow e^{\lla}z\rangle$ embeds into $R$ with $\alpha$ the core geodesic.  For a cusp, normalize the universal covering for the corresponding deck transformation to be $z\rightarrow z+1$.   The cusp region $\mathcal C_{\infty}=\{\Im z\ge 1/2\}/\langle z\rightarrow z+1\rangle$ embeds into $R$.  The collars about short geodesics and cusp regions are mutually disjoint in $R$.  

In the universal cover a Fenchel-Nielsen twist deformation for a single geodesic line $\beta$ is the piecewise isometry self map of $\mathbb H$ with jump discontinuity across $\beta$ given by a hyperbolic transformation stabilizing $\beta$.  A twist deformation of magnitude $t$ offsets the $\beta$ half planes by a relative $t$ units to the right, as measured when crossing $\beta$.  The relative displacement of a combination of twists on disjoint lines is found as follows.  For the displacement of $q$ relative to $p$, consider the twist lines separating $p$ and $q$ (for neither point on a twist line).  There is a partial ordering of lines based on containment of half planes containing $p$.  By definition the $(n+1)^{st}$ line contains the preceding $n$ lines in a common half plane with $p$. The individual twist deformations are normalized to fix $p$.  The combined deformation map of $\mathbb H$ is given by left (post) composition of the individual deformations formed in the order of the lines.  A basic property is that the Fenchel-Nielsen twists on a set of disjoint lines is a commutative group.     

A finite collection of disjoint closed geodesics on a surface $R$ lifts to a locally finite collection in $\mathbb H$ and an equivariant twist mapping is determined on relatively compact sets.  For our purposes it suffices to analyze finite combinations of twists in $\mathbb H$. 

We begin with hyperbolic cylinders and cusp regions.  

\begin{definition} For a hyperbolic cylinder with core geodesic $\gamma$, an opposing twist is a finite combination of weighted Fenchel-Nielsen twists with twist lines orthogonal to $\gamma$ and vanishing magnitude sum.   For a hyperbolic cusp region, a Thurston shear is a finite combination of weighted Fenchel-Nielsen twists with twist lines asymptotic at the cusp and with vanishing magnitude sum.
\end{definition} 

A positive shear corresponds to a right earthquake.  For a Thurston shear an initial piecewise horocycle orthogonal to the twist lines with successive displacements given by the negative weights is deformed to a closed horocycle.  The deformed region is complete hyperbolic with a closed horocycle, consequently is a cusp region. The vanishing magnitude sum condition is required for completeness of the deformed structure.  The condition is noted in \cite[\S 12.3]{Bonshear} and considered in detail in \cite[Chap. 2, \S 4]{Penbk}.  

\begin{lemma}\label{oppbd} The opposing twist deformation of a hyperbolic cylinder is a hyperbolic cylinder.   The core length of the deformed cylinder is bounded uniformly in terms of the initial core length and the twist weights.  For a bounded number of bounded weights, the deformed core length is small uniformly as the initial core length is small.
\end{lemma}

\begin{proof}  Opposing twist lines decompose a cylinder into bands, each isometric to a region between ultra parallel lines in $\mathbb H$.  The twist deformation is given by translations across lines.  The vanishing magnitude sum provides that a deformed cylinder is complete hyperbolic containing ultra parallel bands, consequently is a hyperbolic cylinder.  

We observe that for disjoint weighted twist lines converging, Fenchel-Nielsen twists (normalized with a common fixed region) converge.  For a core length $\ell$, collar twist lines are represented in the band $\{1\le |z|<e^{\ell}\}$ in $\mathbb H$.  For $\ell$ small, the individual twists are close to the twist line $|z|=1$.  
The magnitude sum vanishing provides that for $\ell$ small the combined twist transformation is close to the identity.  In particular for twist weights bounded on a compact set the opposing twist is close to the identity uniformly in $\ell$.  The deformed core length is the translation length of $z\rightarrow ze^{\ell}$ conjugated by the opposing twists. The deformed core length is uniformly small in $\ell$, as desired.  
\end{proof}

Next we make precise the notion of opposing twists converging to a Thurston shear and also note a consequence.   

\begin{definition}
Opposing twists for a sequence of cylinders with core lengths tending to zero geometrically converge to a Thurston shear provided the following.  First, the universal coverings are normalized with the hyperbolic deck transformations for the cylinders converging in the compact open topology for $\mathbb H$ to the parabolic deck transformation for the cusp region.  Second, for a relatively compact open set $K$ in $\mathbb H$ whose projection to the cusp region contains a loop encircling the cusp, the intersection with $K$ of the weighted twist lines for the cylinders converges to the intersection with the weighted Thurston shear lines.
\end{definition}

\begin{figure}[htbp] 
  \centering
  \includegraphics[bb=0 0 538 263,width=3.25in,height=2.8in,keepaspectratio]{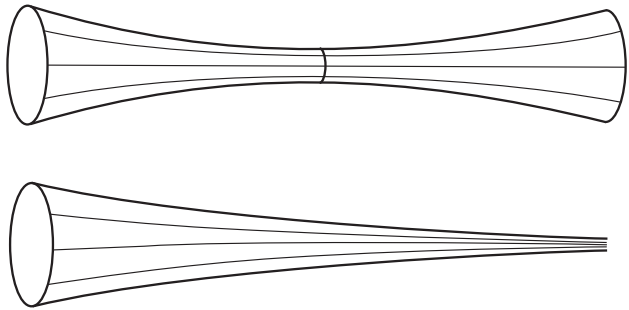}
  \caption{A hyperbolic cylinder with geodesics orthogonal to the core geodesic and a cusp region with geodesics asymptotic at the cusp.}
  \label{fig:collarscusps}
\end{figure}

\begin{lemma}\label{oppconv} Consider hyperbolic cylinders converging to a cusp region with opposing twists geometrically converging to a Thurston shear.  A normalization by $\mathbb H$ isometries of the twist deformation maps of $\mathbb H$ converges to the Thurston shear in the compact open topology for $\mathbb H$.   
\end{lemma}

\begin{proof} Convergence of lines intersecting a given relatively compact set in $\mathbb H$ provides convergence on any compact set.  As noted convergence of weighted lines in $\mathbb H$ provides that suitably normalized deformation maps converge in the compact open topology.
\end{proof}

\section{Chatauby convergence and opening cusps}\label{Chat}

The points of  Teichm\"{u}ller space $\mathcal T$ are equivalence classes 
$\{(R,f)\}$ of Riemann surfaces with reference homeomorphisms $f:F\rightarrow R$ from a reference surface.  The {\em complex of curves} $C(F)$ is defined as follows.  The vertices of
$C(F)$ are the free homotopy classes of homotopically nontrivial, non peripheral,
simple closed curves on $F$.  A $k$-simplex consists of $k+1$ homotopy classes of mutually disjoint simple closed curves.  For surfaces of genus $g$ and $n$ punctures, a maximal
set of mutually disjoint simple closed curves, a {\em partition},
has $3g-3+n$ elements.  The mapping class group acts on the complex $C(F)$.

The Fenchel-Nielsen coordinates for $\mathcal T$ are given in terms of geodesic-lengths and
lengths of auxiliary geodesic segments, \cite{Abbook,Busbook,ImTan}.  A partition
 $\mathcal{P}=\{\alpha_1,\dots,\alpha_{3g-3+n}\}$ decomposes the
reference surface $F$ into $2g-2+n$ components, each homeomorphic to
a sphere with a combination of three discs or points removed.  A homotopy marked
Riemann surface $(R,f)$ is likewise decomposed into pants by the geodesics
representing the elements of $\mathcal P$.  Each component pants, relative to its hyperbolic
metric, has a combination of three geodesic boundaries and cusps.  For each
component pants, the shortest geodesic segments connecting boundaries determine
 designated points  on each boundary.  For each geodesic $\alpha$ in the pants
decomposition, a twist parameter $\tau_{\alpha}$ is defined as the displacement along
the geodesic between designated points, one  for each side of the geodesic.
For marked Riemann surfaces close to an initial reference marked Riemann surface, the
displacement $\tau_{\alpha}$ is the distance between the designated points; in
general the displacement is the analytic continuation (the lifting) of the
distance measurement.  For $\alpha$ in $\mathcal P$ define the {\em
Fenchel-Nielsen angle} by $\vartheta_{\alpha}=2\pi\tau_{\alpha}/\ell_{\alpha}$.
The Fenchel-Nielsen coordinates for Teichm\"{u}ller space for the
decomposition $\mathcal P$ are
$(\ell_{\alpha_1},\vartheta_{\alpha_1},\dots,\ell_{\alpha_{3g-3+n}},\vartheta_{\alpha_{3g-3+n}})$.
The coordinates provide a real analytic equivalence of $\mathcal T$ to
$(\mathbb{R}_+\times \mathbb{R})^{3g-3+n}$, \cite{Abbook,Busbook,ImTan}.

A partial compactification, the {\em augmented Teichm\"{u}ller space} $\Tbar$, is introduced by extending the range of the Fenchel-Nielsen parameters.  The added points correspond to unions of hyperbolic surfaces with formal pairings of cusps.  The interpretation of {\em length
vanishing} is the key ingredient.  For an $\ell_{\alpha}$ equal to zero, the
angle $\vartheta_{\alpha}$ is not defined and in place of the geodesic for
$\alpha$ there appears a pair of cusps; the reference map $f$ is now a homeomorphism of
$F-\alpha$ to a union  of hyperbolic surfaces (curves parallel to $\alpha$
map to loops encircling the cusps).  The parameter space for a pair
$(\ell_{\alpha},\vartheta_{\alpha})$ will be the identification space $\mathbb{R}_{\ge
0}\times\mathbb{R}/\{(0,y)\sim(0,y')\}$.  More generally for the partition 
$\mathcal P$, a frontier set $\mathcal{T}(\mathcal P)$ is added to the
Teichm\"{u}ller space by extending the Fenchel-Nielsen parameter ranges: for
each $\alpha\in\mathcal{P}$, extend the range of $\ell_{\alpha}$ to include
the value $0$, with $\vartheta_{\alpha}$ not defined for $\ell_{\alpha}=0$.  The
points of $\mathcal{T}(\mathcal P)$ in general parameterize unions of Riemann
surfaces with each $\ell_{\alpha}=0,\,\alpha\in\mathcal{P},$ specifying a pair
of cusps. 

We present an alternate description of the frontier points in terms of representations of groups and the Chabauty topology.  A Riemann surface with punctures and hyperbolic metric is uniformized by a cofinite subgroup $\Gamma\subset PSL(2;\mathbb R)$.  A puncture corresponds to the $\Gamma$-conjugacy class of a maximal parabolic subgroup.  In general, a Riemann surface with punctures corresponds to the $PSL(2;\mathbb R)$ conjugacy class of a tuple $(\Gamma,\langle\Gamma_{01}\rangle ,\dots,\langle\Gamma_{0n}\rangle )$ where $\langle\Gamma_{0j}\rangle $ are the maximal parabolic classes and a labeling for punctures is a labeling for conjugacy classes.  A {\em Riemann surface with nodes} $R'$ is a finite collection of $PSL(2;\mathbb R)$ conjugacy classes of tuples 
$(\Gamma^\ast,\langle\Gamma_{01}^\ast\rangle ,\dots,\langle\Gamma_{0n^\ast}^\ast\rangle )$ with a formal pairing of certain maximal parabolic classes.  The conjugacy class of a tuple is called a {\em part} of $R'$.  The unpaired maximal parabolic classes are the punctures of $R'$ and the genus of $R'$ is defined by the relation $Total\ area=2\pi(2g-2+n)$.  A cofinite $PSL(2;\mathbb R)$ injective representation of the fundamental group of a surface is topologically allowable provided peripheral elements correspond to peripheral elements.  A point of the Teichm\"{u}ller space $\mathcal T$ is given by the $PSL(2;\mathbb R)$ conjugacy class of a topologically allowable injective cofinite representation of the fundamental group $\pi_1(F)\rightarrow\Gamma\subset PSL(2;\mathbb R)$.  For a simplex $\sigma$, a point of the corresponding frontier space $\mathcal T(\sigma)\subset\Tbar$ is given by a collection $\{(\Gamma^\ast,\langle\Gamma_{01}^\ast\rangle ,\dots,\langle\Gamma_{0n^\ast}^\ast\rangle )\}$ of tuples with: a bijection between $\sigma$ and the paired maximal parabolic classes; a bijection between the components $\{F_j\}$ of $F-\sigma$ and the conjugacy classes of parts $(\Gamma^j,\langle\Gamma_{01}^j\rangle ,\dots,\langle\Gamma_{0n^j}^j\rangle )$ and the $PSL(2;\mathbb R)$ conjugacy classes of topologically allowable isomorphisms $\pi_1(F_j)\rightarrow\Gamma^j$, \cite{Abdegn, Bersdeg}. We are interested in geodesic-lengths for a sequence of points of $\mathcal T$ converging to a point of $\mathcal T(\sigma)$.  The convergence of hyperbolic metrics provides that for closed curves of $F$ disjoint from $\sigma$, geodesic-lengths converge, while closed curves with essential $\sigma$ intersections have geodesic-lengths tending to infinity, \cite{Bersdeg, Wlhyp}.  

We refer to the Chabauty topology to describe the convergence for the $PSL(2;\mathbb R)$ representations.   Chabauty introduced a topology of geometric convergence for the space of discrete subgroups of a locally compact group, \cite{Chb}.   A neighborhood of $\Gamma\subset PSL(2;\mathbb R)$ is specified by a neighborhood $U$ of the identity in $PSL(2;\mathbb R)$ and a compact subset $K\subset PSL(2;\mathbb R)$.  A discrete group $\Gamma'$ is in the neighborhood $\mathcal N(\Gamma,U,K)$ provided $\Gamma'\cap K\subseteq\Gamma U$ and $\Gamma\cap K\subseteq\Gamma'U$.  The sets $\mathcal N(\Gamma,U,K)$ provide a neighborhood basis for the topology.   The $PSL(2;\mathbb R)$ topology coincides with the induced compact open topology for transformations of $\mathbb H$.  
Important for the present considerations is the following convergence characterization.  A sequence of points of $\mathcal T$ converges to a point of $\mathcal T(\sigma)$, provided for each component $F_j$ of $F-\sigma$, there exist $PSL(2;\mathbb R)$ conjugations such that restricted to $\pi_1(F_j)$ the corresponding representations converge element wise to $\pi_1(F_j)\rightarrow\Gamma^j$, \cite[Thrm. 2]{HrCh}.  

We now consider a Riemann surface $R$ with cusps and data $\sum\mathfrak b_j\widehat\beta_j$ for a Thurston shear.  The data is a weighted sum of disjoint simple ideal geodesics, geodesics with endpoints at infinity in the cusps.  The weighted sum of segments entering each cusp vanishes.  Double the surface across its cusps; consider the union of $R$ and its conjugate surface $\bar R$ with the reflection symmetry $\rho$ for the pair.  For the geodesic $\widehat\beta_j$,  we write $\beta_j$ for the union $\widehat\beta_j\cup\rho(\widehat\beta_j)$.  To open cusps, given $\epsilon$ positive, remove the area $\epsilon$ horoball at each cusp and glue the remaining surfaces by the map $\rho$ to obtain a compact surface $R_{\epsilon}$.  The surface $R_{\epsilon}$ has a reflection symmetry (also denoted $\rho$) and smooth simple closed curves obtained from surgering the $\beta_j$ (also denoted $\beta_j$).  The construction provides a homeomorphism from a reference surface $F$ to $R_{\epsilon}$ for $\epsilon$ positive and the simplex $\sigma$ of short curves for $F$ is given by the $\epsilon$ horocycles.  Standard comparison estimates for metrics provide that for the uniformization hyperbolic metric, the simplex is realized by short geodesics with lengths tending to zero with $\epsilon$.  The comparison estimates also provide that on the complement of prescribed area collars about the short geodesics, the $R_{\epsilon}$ hyperbolic metrics converge $C^{\infty}$ to the hyperbolic metric of $R\cup\bar R$, \cite{Wlhyp}.  The uniformization groups $\Gamma(R_{\epsilon})$ for the $R_{\epsilon}$, Chatauby converge to the uniformization pair $\Gamma(R),\,\Gamma(\bar R)$, relative to $F$ and the horocycle simplex $\sigma$.  The uniqueness of geodesics and convergence of hyperbolic metrics provide that the geodesics $\tilde \beta_j$ in the free homotopy classes $\beta_j$ converge uniformly on $\sigma$ collar complements to $\widehat \beta_j\cup\rho(\widehat\beta_j)$ on $R\cup\bar R$.  

We are ready to compare the effect of the Thurston shear $\sum\mathfrak b_j(\widehat\beta_j\cup-\rho(\widehat\beta_j))$ on $R\cup\bar R$ to the effect of the opposing twist $\sum\mathfrak b_j\tilde\beta_j$ on the hyperbolic metric of $R_{\epsilon}$.  The reflection $\rho$ reverses orientation and notions of left/right; even though $\bar R$ is the mirror image, we require regions to move in the same direction by a twist; the minus sign provides the desired effect.  Opposing twist deformations do not preserve the reflection symmetry.  As a preliminary matter, we note from Lemma \ref{oppbd} for weights bounded, the opposing twist of $R_{\epsilon}$ has small geodesic lengths bounded in terms of $\epsilon$.   Twisting  $R_{\epsilon}$ defines a family close to the frontier $\mathcal T(\sigma)$.  We observe the following.

\begin{lemma}\label{opclose} For $\epsilon$ small and weights bounded, the opposing twist $\sum\mathfrak b_j\tilde\beta_j$ of $R_{\epsilon}$ is Chatauby close to the Thurston shear 
$\sum\mathfrak b_j(\widehat\beta_j\cup-\rho(\widehat\beta_j))$ of $R\cup\bar R$.  Furthermore,  the infinitesimal opposing twist is close to the infinitesimal Thurston shear in the sense of infinitesimal variations of $PSL(2;\mathbb R)$ representations.
\end{lemma}
\begin{proof}  In brief the convergence of metrics provides for the compact open convergence of the twist/shear lines on $\mathbb H$, which in turn provides for the element wise convergence of representations.   By construction of $R_{\epsilon}$, for the components $F_j$ of $F-\sigma$, the representations $\pi_1(F_j)$ into $PSL(2;\mathbb R)$ converge element wise and the twist lines compact open converge to shear lines.  Choose generators for the limiting representations and a relatively compact open set $U\subset\mathbb H$, such that $CU\cap U\ne\varnothing$ for each generator $C$. 
For $\epsilon$ small, the same elements generate the representations of $\pi_1(F_j)$ and satisfy the non empty translate intersection condition.  The representations are completely determined by their action on $U$.  A twist/shear map $\tau$ of $\mathbb H$ induces a variation of a representation by varying a transformation $B$ by the conjugation $\tau B\tau^{-1}$. Only a finite number of twist/shear lines intersect $U$.  The $PSL(2;\mathbb R)$ normalized combined twist is given by finite ordered compositions as described above.  By metric convergence, as $\epsilon$ tends to zero, on $U$ the twist lines converge uniformly and the twists converge uniformly to shears and thus the representations of the finite number of generators converge.  The representations are element wise uniformly close in $\epsilon$.  To consider the infinitesimal variations, we introduce a parameter $t$ for $t\sum\mathfrak b_j\tilde\beta_j$ and $t\sum\mathfrak b_j(\widehat\beta_j\cup-\rho(\widehat\beta_j))$.  The considerations provide that the initial infinitesimal variations of the generators are also close in $\epsilon$.  The infinitesimal variations of the representations are determined on generators.   
\end{proof}       

\section{Infinitesimal Thurston shears and opposing twists}\label{results}

We are interested in geodesic-length gradients.  A thick-thin decomposition of hyperbolic surfaces is determined by a positive constant.  The thin subset consists of those points with injectivity radius at most the positive constant; for a constant at most unity the thin subset is a disjoint union of collars and horoballs  \cite{Busbook}.  Surface representations into $PSL(2;\mathbb R)$ are Chatauby close precisely when their thick subsets are Gromov-Hausdorff close.  For a sequence of hyperbolic surfaces with certain geodesic-lengths tending to zero, we are interested in the magnitude and convergence of geodesic-length gradients $\grad\lla$ for geodesics $\alpha$ crossing the short geodesic-length collars. 

Applications of convergence of surfaces and gradients include generalizing the Gardiner formula, Theorem \ref{Gardtheta}, to balanced sums of ideal geodesics and generalizing twist length duality (\ref{wpdual}) to Thurston shears and balanced sums of ideal geodesics.  The basic matter is to understand the effect of Chatauby convergence for sums of the basic differential $\Omega^2$ from Section \ref{glfs}.  We begin with convergence of hyperbolic transformations of $\mathbb H$. 

A hyperbolic transformation with translation length $\ell$, fixed points symmetric with respect to the origin and $i$ on its collar boundary is given as
\[
A\,=\,
\begin{pmatrix}
\cosh \ell/2 & 1/\ell\,\sinh \ell/2 \\
\ell\,\sinh \ell/2 & \cosh \ell/2
\end{pmatrix} 
\]
($i$ is distance $\log1/\ell$ to the $A$ axis with endpoints $\pm 1/\ell$).  As $\ell$ tends to zero, $A$ converges to the parabolic transformation
\[
\begin{pmatrix}
1 & 1/2 \\ 0 & 1
\end{pmatrix}.  
\]

We consider a Chatauby converging sequence of surfaces with short length core geodesics and a crossing geodesic intersecting the core geodesics orthogonally.   
\begin{figure}[htbp] 
  \centering
  \includegraphics[bb=0 0 460 152,width=4in,height=1.2in,keepaspectratio]{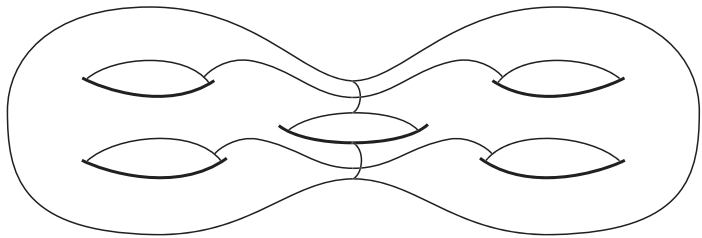}
  \caption{A symmetric compact surface with crossing and core geodesics.}
  \label{fig:crossing}
\end{figure}
\noindent A crossing geodesic intersects collars and core geodesics.  Given a segment of a crossing geodesic $\alpha$ in a thick region, normalize the universal coverings so that the segment lifts to a segment along the imaginary axis with highest point at $i$.  Extend the segment by including the arcs that connect to core geodesics (the added arcs cross half collars).  A core geodesic intersecting $\alpha$ lifts to a geodesic orthogonal to the imaginary axis.  The figures for the universal covers of the surface, Figure \ref{fig:fundregn}, and the Chatauby limit, Figure \ref{fig:fundcusp}, are as follows.   In the figures the collar lift and its limit are shaded.  In Figure \ref{fig:fundregn}, the left and right circular arcs orthogonal to the baseline bound a fundamental domain for a core geodesic transformation.  

\begin{figure}[htb] 
  \centering
  \includegraphics[bb=0 0 519 360,width=2.9in,height=2.01in,keepaspectratio]{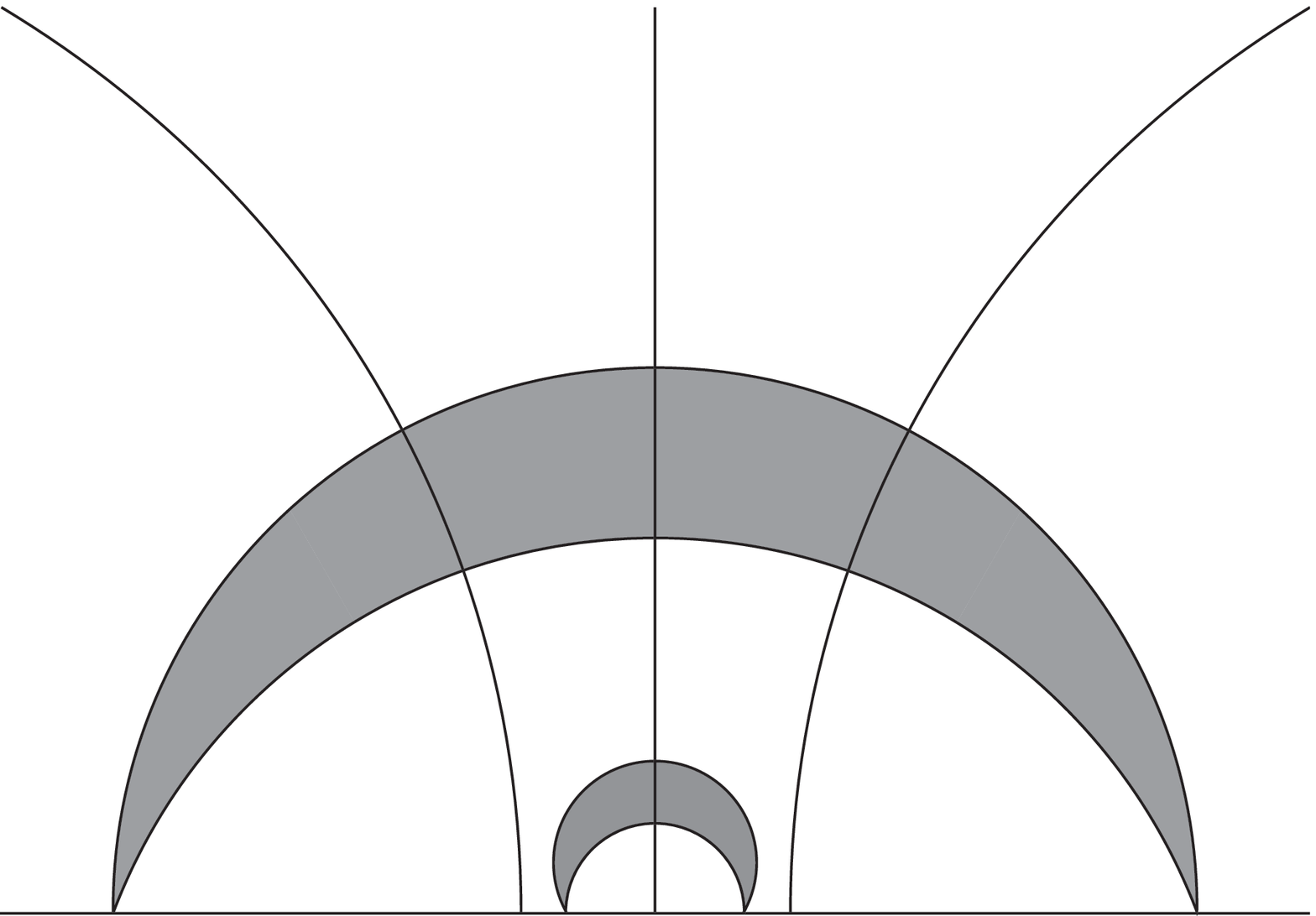}
  \caption{Crossing and core geodesics.  The vertical line is the lift of the crossing geodesic.  The two semi circles orthogonal to the baseline are consecutive lifts of the core geodesic.  The left and right circular arcs bound a fundamental domain for the hyperbolic transformation stabilizing the larger semi circle.  The shaded sectors are lifts of half collars for the core geodesic.  The region bounded by the shaded sectors and the circular arcs covers a region containing a component of the thick subset of the surface.}
  \label{fig:fundregn}
\end{figure}

\begin{figure}[htb] 
  \centering
  \includegraphics[bb=0 0 508 382,width=3.15in,height=2.37in,keepaspectratio]{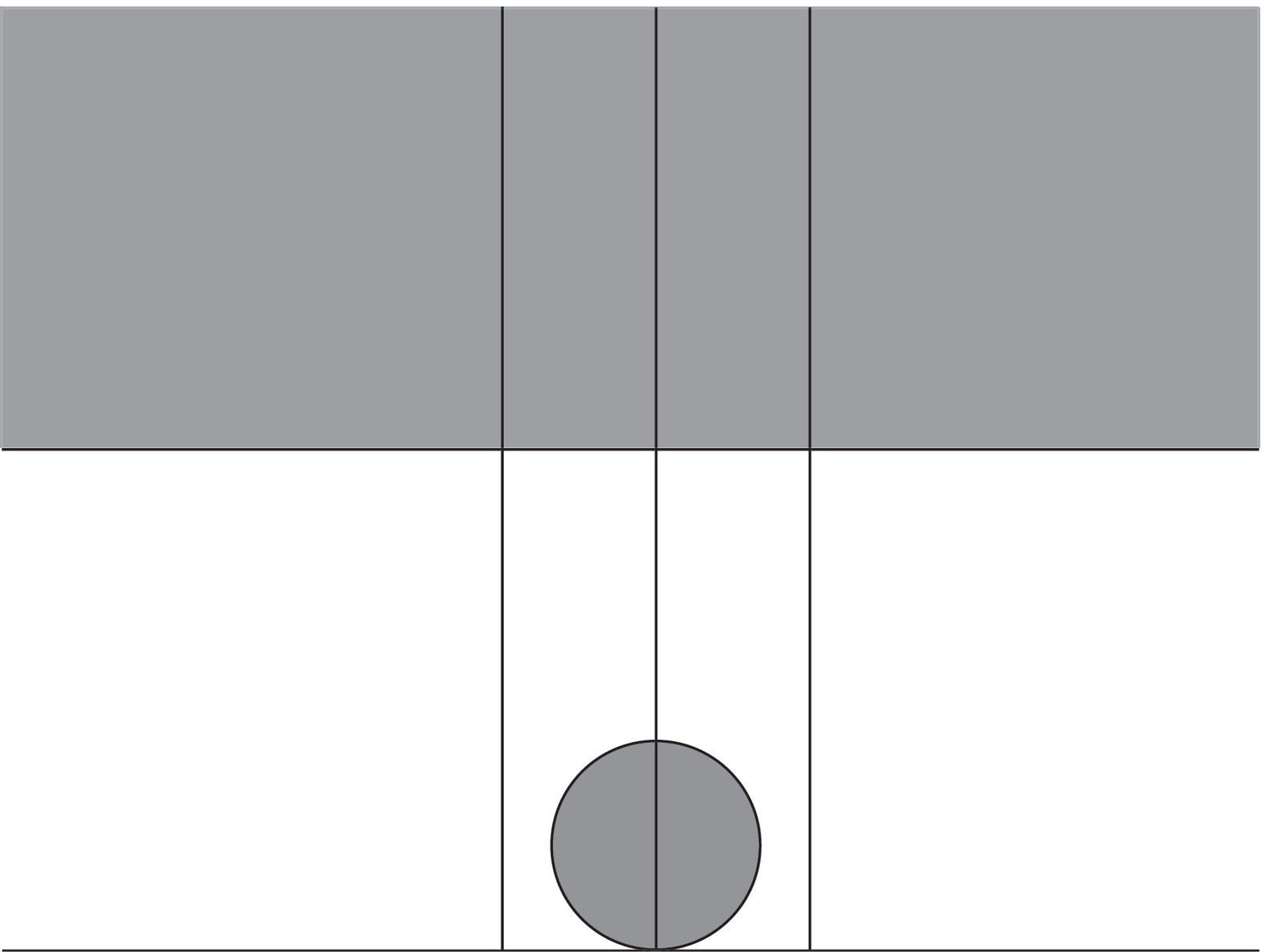}
  \caption{An ideal geodesic and horoballs.  The central vertical line is the lift of the ideal geodesic connecting cusps.  The left and right vertical lines bound a fundamental domain for the parabolic transformation stabilizing infinity.  The shaded sectors are horoballs about the cusps.  The region bounded by the shaded sectors and the vertical lines covers a region containing the thick subset of the surface. }
  \label{fig:fundcusp}
\end{figure}

\noindent Chatauby convergence provides that the original segments on the crossing geodesic $\alpha$ have length bounded and it is standard that collar boundaries converge to horocycles.  Figure \ref{fig:fundcusp} is the limit of a sequence of Figures \ref{fig:fundregn} with upper, respectively lower, shaded regions converging to upper, respectively lower, shaded regions.   The crossing geodesic limits to an ideal geodesic connecting cusps.

\begin{definition}  For an ideal geodesic $\alpha$, we write
\[
d\lla\,=\,\frac{2}{\pi}\sum_{C\in\Gamma}\Omega_{pq}^2(Cz)
\]
for the infinite series, where $p,q$ are endpoints of a lift of $\alpha$ to $\mathbb H$.
\end{definition}
  
\begin{lemma}\label{idglfb}  For a surface $R$ with cusps and an ideal geodesic $\alpha$, the infinite series $d\lla$ converges.  As above, consider surfaces $R_{\epsilon}$ with reflection symmetries obtained by doubling $R$ across its cusps and opening cusps to obtain short length core geodesics.  Consider that an ideal geodesic $\alpha$ on $R$ is approximated on thick subsets by closed core orthogonal geodesics $\alpha_{\epsilon}$ on $R_{\epsilon}$.  There is a  Chatauby neighborhood $\mathcal U$ of $R\cup\bar R$ such that for $R_{\epsilon}\in\mathcal U$, on thick subsets the harmonic Beltrami differentials $d\ell_{\alpha_{\epsilon}}(ds^2)^{-1}$ and $d\lla(ds^2)^{-1}$ are uniformly bounded and are uniformly close.
\end{lemma}
\begin{proof}The $d\lla$ series are bounded by area integrals as follows. We first consider regions.  In Figure \ref{fig:fundregn}, the unshaded region in $\mathbb H$ between the shaded crescents, by normalization, lies below the line $\Im z=1$ and outside a circle tangent to $\mathbb R$ at $0$.  The integral of $|\Omega_{0\infty}|^2=dr/r\,d\theta$ for $z=re^{i\theta}$ over the unshaded region is bounded by the integral over the region between the shaded sectors in Figure \ref{fig:fundcusp}
\[
\int^{\pi}_0\int^{\csc \theta}_{a\sin\theta}\frac{dr}{r}d\theta\,=\,\int_0^{\pi}\log\frac{\csc^2\theta}{a}\,d\theta\,=\,2\pi\log 2\,-\,\pi\log a.
\]
On a thick region of a surface a holomorphic quadratic differential satisfies a mean value estimate in terms of the integral over a hyperbolic metric ball of a radius $r_0$ at most the injectivity radius.  The thick regions of $R_{\epsilon}$ and $R$ are contained in the projection of the indicated unshaded regions in Figures \ref{fig:fundregn} and \ref{fig:fundcusp}.   By the standard unfolding, the absolute values of $d\ell_{\alpha_{\epsilon}}(ds^2)^{-1}$ and $d\lla(ds^2)^{-1}$ at a thick point are bounded by the integral of  $|\Omega_{0\infty}|^2$ over the disjoint union of $r_0$ balls about the orbit of the lifted point in the unshaded region \cite[Chapter 8]{Wlcbms}.  By the above considerations, the integrals are uniformly bounded, establishing the first result.  

For the second conclusion, given $\delta$ positive, choose a relatively compact set $K$ in the Figure \ref{fig:fundregn} region between shaded crescents, such that the integral of $|\Omega_{0\infty}|^2$ over the complement between the shaded crescents is bounded by $\delta$.  The sum of evaluations of $\Omega_{0\infty}^2$ at points not in $K$ is bounded by $\delta$ by a mean value estimate.  Chatauby convergence provides convergence for the sum of evaluations of $\Omega_{0\infty}^2$ for the orbit points in $K$. Boundedness and convergence are established.  
\end{proof}

\begin{example} \textup{The ideal geodesic series $d\lla$ for a hyperbolic cusp.}
For a cusp uniformized at infinity with integer translation group then the sum over the group is
\[
\sum_{C\in\Gamma_{\infty}}\Omega_{0\infty}^2(Cz)\,=\,\sum_{n\in\mathbb Z}\frac{dz^2}{(z-n)^2}.
\]
The formula for the integer sum gives $d\lla=2\pi\csc^2\pi z\,dz^2$.  From the above lemma, for a hyperbolic cylinder the series 
$d\lla$ approximates the cosecant squared in the compact open topology of $\mathbb H$.
\end{example}

We now combine considerations to obtain a uniform majorant for an opposing sum of twists and gradients of geodesic-length functions.  The majorant is the necessary ingredient for general limiting arguments.  We codify the situation as follows.
\begin{definition}
A crossing configuration is a compact surface with reflection symmetry with fixed locus a finite union of small length core geodesics $\gamma$ and no other geodesics having small length.  A crossing geodesic $\alpha$ is symmetric with respect to the reflection with two intersections with the core geodesics.  For a crossing configuration, a sum 
$\sum\mathfrak a_j\ell_{\alpha_j}$ of crossing geodesics length functions is balanced provided for each core geodesic $\gamma$, the weighted intersection number $\sum\mathfrak a_j\#(\alpha_j\cap\gamma)$ vanishes.  For a surface with cusps, a formal sum $\sum\mathfrak a_j\ell_{\alpha_j}$ of ideal geodesics length functions is balanced provided at each cusp the weighted intersection number $\sum\mathfrak a_j\#(\alpha_j\cap h)$ with each small closed horocycle $h$ vanishes.
\end{definition}

Balanced is the precedent to the condition of the weight sum vanishing for each cusp for a Thurston shear. To prepare for a convergence argument, we first consider the distribution of mass of a harmonic Beltrami differential.   
 
\begin{lemma}\label{bglfb} A balanced sum $\sigma=\sum \mathfrak a_j\grad\ell_{\alpha_j}$ of gradients for a crossing configuration is bounded as follows. On the thick subset the absolute value $|\sigma|$ is uniformly bounded.  On a core geodesic $\gamma$ collar, uniformized as $1\le|z|\le e^{\llg}$, 
$\ell_{\gamma}\le \theta\le \pi-\ell_{\gamma}$ for $z=re^{i\theta}\in\mathbb H$, the balanced sum $\sigma$ is bounded as
\[
O\big((\ell_{\gamma}^3+e^{-2\pi\theta/\ell_{\gamma}}\ +\ e^{2\pi(\theta-\pi)/\ell_{\gamma}})\ell_{\gamma}^{-2}\sin^2\theta\big).
\]
The bounding constants depend only on the number of crossing geodesics, the norm of the weights and a choice of Chatauby neighborhood for the limiting cusped surface.
\end{lemma}
\begin{proof}A general bound for a harmonic Beltrami differential on a $\gamma$ collar is
\begin{equation}\label{collbd}
|\mu|\quad \mbox{is}\quad O\big(\big(|(\mu,\grad\log\llg)|\ +\ (e^{-2\pi\theta/\ell_{\gamma}}\ +\ e^{2\pi(\theta-\pi)/\ell_{\gamma}})\llg^{-2}\big)\sin^2\theta\, M\big)
\end{equation}
for $M$ the maximum of $\mu$ on the collar boundary \cite[Prop. 6]{Wlcurv}.  We use Theorem \ref{gradpr} to bound the pairings $\langle\grad\ell_{\alpha_j},\grad\llg\rangle$.  By setup the crossing and core geodesics are orthogonal.  Each core geodesic intersection contributes $-2$ to the pairing evaluation.  From the balanced hypothesis, the weighted sum of intersection contributions vanishes.  Each remaining term of the evaluation involves a connecting geodesic segment that crosses the $\gamma$ half collar; the width of the half collar is $-\log\llg$.  For large distance, the formula summand $\mathcal R$ is approximately $e^{-2d(\gamma,\alpha)}$.  In \cite[Chap. 8]{Wlcbms} we showed that the sum of distances from $\alpha$ to the $\gamma$ collar boundary is uniformly bounded.  It follows that the contribution $\llg^2$ of the half collar width can be factored out of each summand. The sum evaluation is $O(\llg^2)$, the desired bound.  Lemma \ref{idglfb} provides the desired bound for $\sigma$ on the thick subset.  
\end{proof}

\section{The symplectic geometry of lengths}\label{sympgeom}

There is a length interpretation for a balanced sum $\mathcal A=\sum \mathfrak a_j \ell_{\alpha_j}$ of ideal geodesics length functions as follows.  Let $\mathcal H$ be a neighborhood of the cusps given as a union of small horoballs, one at each cusp.  The length $L(\mathcal A)$ of the balanced sum is the sum with weights $\mathfrak a_j$ of lengths of segments $\alpha_j\cap (R-\mathcal H)$.  The balanced condition provides that the length does not depend on the choice of horoball neighborhood $\mathcal H$.  For a crossing configuration the length of a balanced sum $L(\mathcal A)$ is defined in the corresponding manner.  In the crossing case, the value $L$ coincides with the sum of geodesic lengths. 

The length $L(\mathcal A)$ of a balanced sum is a generalization of the $R$-length of a transverse cocycle. The balanced condition at cusps is discussed in \cite[\S 12.3]{Bonshear}, where it is noted that the condition provides a well-defined notion of length.  The definition in terms of horoballs shows that the length $L(\mathcal A)$ is given as $\sum2 \mathfrak a_j\log\lambda_{\alpha_j}$ for the $\lambda$-lengths of the ideal geodesics and a decoration.  An example of a balanced sum is a shear coordinate $\sigma_*$, see formula (\ref{shear}); the sum is balanced at each vertex of the quadrilateral of Figure
 \ref{fig:diamond}.  A second example comes directly from the shear coordinates of Riemann surfaces.  By Theorem \ref{Penthrm}, the sum $\sum\sigma_j\ell_{\alpha_j}$ is balanced since the sum of shear coordinates around each cusp vanishes. The adjustment of a factor of $2$ to our formulas as detailed in \cite[\S 5]{Wlcusps} is included in the following.  

\begin{theorem}\label{twlth} For a surface $R$ with cusps and a balanced sum $\mathcal A=\sum \mathfrak a_j \ell_{\alpha_j}$ of ideal geodesics length functions, the length $L(\mathcal A)$ is a differentiable function on the Teichm\"{u}ller space of $R$ with 
\[
dL(\mathcal A)\ =\ \sum \mathfrak a_jd\ell_{\alpha_j}\,\in\,Q(R).
\]
The formal sum $\sum \mathfrak a_j\alpha_j$ is data for an infinitesimal Thurston shear 
$\sigma_{\mathcal A}$ with
\[
\sigma_{\mathcal A}\ =\ \frac{i}{2}\sum \mathfrak a_j\grad \ell_{\alpha_j}.
\]
The WP twist-length duality
\[
2\omega_{WP}(\ ,\sigma_{\mathcal A})\ =\ dL(\mathcal A)
\]
is satisfied.  In particular, the Thurston infinitesimal shear $\sigma_{\mathcal A}$ is a WP symplectic vector field with Hamiltonian potential function $L(\mathcal A)/2$. 
\end{theorem}
\begin{proof}  We first observe that $L$ is a differentiable function on the $PSL(2;\mathbb R)$ representation space.  For the reference surface $F$, a simple loop $\delta\in\pi_1(F)$ about the cusp has representation into $PSL(2;\mathbb R)$ a parabolic element that generates a maximal parabolic subgroup.   Prescribing an area value (at most unity) for the quotient of a horoball by the maximal parabolic subgroup determines a horoball and horocycle.  (The prescription is equivalent to a choice of decoration in the Penner approach \cite{Pencell,Penbk}.)  For a pair of elements of $\pi_1(F)$ defining distinct maximal parabolic subgroups, the distance between the prescribed horocycles is a smooth function of the $PSL(2;\mathbb R)$ representation.  The length $L$ is a sum of distances between horocycles and hence a smooth function.  The differential $dL$ is an element of $Q(R)$.  In particular the integral of the element over small neighborhoods of the cusps is small.  The construction of the function and its differential is also valid for the distance between collar boundaries.   

Consider a sequence of compact surfaces $R_{\epsilon}$ with reflection symmetries obtained by doubling and opening the cusps of $R$.  From Lemma \ref{idglfb}, on thick subsets, the differentials of geodesic-lengths converge uniformly to differentials for ideal geodesics.  From Lemma \ref{bglfb},  for a balanced sum, the sum of differentials is uniformly bounded in each core collar; the integral of the sum is uniformly small over small area collars.  
As $R_{\epsilon}$ limits to $R$, the distance between collar boundaries limits to the distance between horocycles.  And for closed geodesics $\beta$ contained in the thick subsets, the Fenchel-Nielsen twists on $\beta\cup\rho(\beta)$ of $R_{\epsilon}$ converge to the twist of $R\cup\bar R$ and the twist derivatives of distance converge.  The considerations of Chatauby convergence and Lemmas \ref{idglfb} and \ref{bglfb} can be applied for the Fenchel-Nielsen twists on $\beta\cup\rho(\beta)$.   The conclusion is again that the gradient pairing integrals over small area collars and small area horoballs are uniformly small.  It follows that the pairing for a balanced sum length differential and twist converges to the limiting pairing as $\epsilon$ tends to zero.  The derivative of length converges to the derivative of length.  Reflection-even twists span the reflection-even tangent space.  The $dL$ formula is established.

The considerations for infinitesimal Thurston shears are analogous.  The deformation is smooth and by Lemma \ref{opclose} the infinitesimal deformation is a limit of opposing twists.  The opposing twists satisfy $\sum \mathfrak a_j t_{\alpha_j}=i/2\sum \mathfrak a_j\grad\ell_{\alpha_j}$ on the side of $R_{\epsilon}$ that limits to $R$.  We find the $\epsilon$ tending to zero limit by Lemmas \ref{idglfb} and \ref{bglfb}.  The conclusions follow. 
\end{proof} 

We remark that symmetry is basic to considering the $R_{\epsilon}$ to $R$ limit of the tangent-cotangent pairing.  With respect to the reflection $\rho$, the differential of the length $L(\mathcal A)$ is even, while the opposing twist and its limit are odd.  Also the 
K\"{a}hler form is odd since the reflection reverses orientation for surface integration. The above duality relation $2\omega_{WP}(\ ,\sigma_{\mathcal A})=dL(\mathcal A)$ is established for reflection even tangents of $R\cup\bar R$ and cannot be applied to evaluate a shear pairing $\omega_{WP}(\sigma_{\mathcal B},\sigma_{\mathcal A})$. 

To evaluate the pairing of Thurston shears, we introduce an elementary alternating $2$-form for coefficients summing to zero. For a balanced sequence $\{a_j\}_{j=1}^p$, we consider the partial sums $A_0=0, A_k=\sum_{j=1}^ka_j,\,1\le k\le p,$ where by hypothesis $A_p=0$.  We introduce a pairing for balanced sequences
\begin{equation}\label{2form}
\omega(\{a_j\},\{b_j\})\,=\,\frac12 \sum_{j=1}^p(A_j+A_{j-1})b_j.
\end{equation}
We explain that the pairing depends only on the joint cyclic ordering of the sequences and that the pairing is alternating.  A cyclic shift in the index $j, 1\le j\le p,$ has the effect of adding a constant to the partial sums $A_j,0\le j\le p$. The balanced condition for the sequence $\{b_j\}$ provides that the pairing is unchanged.  For the alternating property, we have summation by parts for balanced sequences $\{f_j\}$ and $\{g_j\}$ with partial sums $F_k$ and $G_k$
\[
\sum_{k=m}^{n-1}F_kg_{k+1}\,=\,F_nG_n\,-\,\sum_{k=m}^nG_kf_k.
\]
In particular we have that
\[
\sum_{j=1}^pA_jb_j\,=\,A_pB_p\,-\,\sum_{j=1}^{p-1}B_ja_{j+1}\,=\,-\,\sum_{j=1}^pB_{j-1}a_j
\]
and 
\[
\sum_{j=1}^pA_{j-1}b_j\,=\,A_pB_p\,-\,\sum_{j=1}^pB_ja_j\,=\,-\,\sum_{j=1}^pB_ja_j
\]
using that $A_0, A_p, B_0$ and $B_p$ vanish.  The pairing can be written in the alternating form  
\begin{equation}\label{2forma}
\omega(\{a_j\},\{b_j\})\,=\,\frac12 \sum_{j=1}^p A_jb_j-B_ja_j.
\end{equation}
We note that balanced sequences have an interpretation as tangents to the regular $(p-1)$-simplex and $\omega$ an interpretation as a closed $2$-form on the regular simplex. 

The form $\omega$ can be evaluated for a pair of balanced sums for a common set of disjoint ideal geodesics limiting to a cusp.  For balanced sums $\mathcal A=\sum a_j\ell_{\alpha_j}, \mathcal B=\sum b_j\ell_{\alpha_j}$ and a given cusp, consider the geodesic segments limiting to the cusp; some geodesics $\alpha_j$ may not limit to the given cusp and some may have both ends limiting to the cusp.  Choose and label a limiting geodesic as the first and enumerate limiting geodesics in the counterclockwise order about the cusp.  Evaluate the form $\omega$ on the enumerated sequences of weights $\{a_j\}$ and $\{b_j\}$.

\begin{corollary}\label{commshear}
For the balanced sums $\mathcal A=\sum a_j\ell_{\alpha_j}$ and $\mathcal B=\sum b_j\ell_{\alpha_j}$ for a common set of disjoint ideal geodesics, the shear pairing is 
\[
\omega_{WP}(\sigma_{\mathcal A},\sigma_{\mathcal B})\,=\,\frac12 \sigma_{\mathcal A}L(\mathcal B)\,=\,\frac12\sum_{\operatorname{cusps}}\omega(\{a_j\},\{b_j\}), 
\]
The Poisson bracket for the length functions $L(\mathcal A)$ and $L(\mathcal B)$ is
\[
\{L(\mathcal A),L(\mathcal B)\}\,=\,2\sum_{\operatorname{cusps}}\omega(\{a_j\},\{b_j\}).
\]
\end{corollary}

\begin{proof} The shear-length duality comes from Theorem \ref{twlth}.  The first line of equations is established by finding the contribution to the change in the length $L(\mathcal B)$ from the change in the determination of a closed horocycle at a cusp. We refer to the schematic Figure \ref{fig:shearlines} for the basic geometry.  
\begin{figure}[tbp] 
  \centering
  \includegraphics[bb=0 0 502 258,width=3.5in,height=1.8in,keepaspectratio]{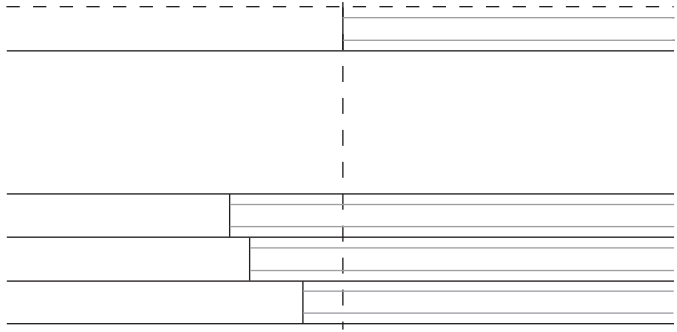}
  \caption{Shear lines at a cusp.  The longer horizontal lines represent ideal geodesics ending at a cusp on the far left; the uppermost and lowermost horizontal lines are identified.  The dotted vertical represents a closed horocycle in the undeformed hyperbolic structure and the shorter solid verticals form a closed horocycle after applying a shear $\sigma_{\mathcal A}$ for the horizontal lines.  The shorter verticals are successively displaced by horizontal increments $-a_1,-a_2,\dots, -a_p$.  The shaded horizontals indicate segments along the upper and lower edges of each ideal geodesic, segments connecting the horocycles of the deformed structure.}
  \label{fig:shearlines}
\end{figure}
To evaluate the change in length and $\omega$, geodesic segments are labeled as described above.  In the $\sigma_{\mathcal A}$ deformed hyperbolic structure, the distance between closed horocycles measured on the upper edge of an ideal geodesic agrees with the distance measured on the lower edge.  We can compute the change in distance by averaging the change for the upper and lower edges.  In Figure \ref{fig:shearlines}, the change in the first distance is $A_1/2$, while the change in the $j^{th}$ distance is $(A_j+A_{j-1})/2$.  For the weighted length $L(\mathcal B)$, the weight for the $j^{th}$ distance is $b_j$.  The change in weighted distance for the given cusp is $\sum(A_j+A_{j-1})b_j/2$, as desired.

We next consider the Poisson bracket.  The non degenerate K\"{a}hler form $\omega_{WP}$ defines an isomorphism from tangent to cotangent spaces and a dual form $\widehat{\omega_{WP}}$.  For the Hamiltonian length functions the Poisson bracket is defined as $\widehat{\omega_{WP}}(dL(\mathcal A),dL(\mathcal B))$.  By duality the pairing is $4\omega_{WP}(\sigma_{\mathcal A},\sigma_{\mathcal B})$.  The final formula follows.
\end{proof}

There is a counterpart to Theorem \ref{thrmE} for the setting of shear coordinates.\footnote{Theorem \ref{thrmE} is formulated for left twists/shears while the present results are formulated for right twists/shears.  The orientation difference explains the interchange of entries when comparing $2$-forms.}  First given an ideal triangulation $\Delta$, Theorem \ref{Penthrm} provides a bijection between balanced sum shears $\sum a_j\mathfrak s_j$ and $\caT$ as follows, for $\mathfrak s_j$ denoting the shear deformations on the $\Delta$ edges.  A {\em basepoint} $R_{\Delta}\in\caT$ in Teichm\"{u}ller space is determined by all shear coordinates vanishing.  The surface $R_{\Delta}$ is constructed by gluing ideal triangles with medians on sides always matching.  Each marked Riemann surface $R\in\caT$ is given uniquely as a balanced sum shear $\sigma_R=\sum a_j(R)\mathfrak s_j$ of the surface $R_{\Delta}$.  We show the balanced sum length functions are linear in the shear coordinates as follows.
\begin{corollary}\label{lpr}
For a balanced sum $\mathcal B=\sum b_j\ell_{\alpha_j}$ of lengths of ideal geodesics of the triangulation $\Delta$ and a marked Riemann surface $R\in\caT$ then 
\[
L(\mathcal B)(R)\,=\,\sum_{\operatorname{cusps}}\omega(\{a_j(R)\},\{b_j\}).
\]
\end{corollary} 
\begin{proof}
First we observe that all balanced sum length functions vanish at $R_{\Delta}$.  Given a balanced sum $\mathcal B=\sum b_j\ell_{\alpha_j}$, consider the double sum of weights 
\[
\sum_{\operatorname{cusps}}\ \sum_{\operatorname{edges\ at \ cusp}}b_{(m,n)},
\]
where the index $m$ enumerates cusps and the index $n$ enumerates half edges entering a cusp.  The balanced sum condition is the vanishing of the inner sums.  Each triangulation edge enters two cusps; the enumeration includes each triangulation edge twice.  Thus the sum of weights of a balanced sum vanishes.  Since the shear coordinates of $R_{\Delta}$ vanish, we can introduce a decoration $\caH$ for $R_{\Delta}$ such that all $h$-lengths have a common value.  It follows that all $\lambda$-lengths have a common value $\lambda_0$.  The length $L(\mathcal B)=\sum b_j2\log\lambda_0$ of the balanced sum vanishes at $R_{\Delta}$.

Given a surface $R$, the path of shears $\sigma_t=t\sum a_j(R)\mathfrak s_j$ connects the surfaces $R_{\Delta}$ and $R$. Corollary \ref{commshear} provides that the $t$-derivative of $L(\caB)$ along the path has the constant value $\sum_{\operatorname{cusps}}\omega(\{a_j(R)\},\{b_j\})$.  Integration in $t$ provides the desired formula.
\end{proof}

By Theorem \ref{Penthrm}, the shear coordinates for the edges of an ideal triangulation provide a continuous immersion into Euclidean space.  In particular the shear coordinates for appropriate subsets of edges provide continuous coordinates for Teichm\"{u}ller space.  A procedure determining appropriate subsets of edges is given in the proof of Lemma \ref{edgebasis} below.  From Theorem \ref{Penthrm}, for a subset of shear coordinates without linear relations, the differentials of the coordinates are generically linearly independent.  Furthermore from Corollary \ref{commshear}, for a subset of shear coordinates without linear relations there are sets of balanced sum length functions with constant full rank Poisson bracket pairing.  It follows from the pointwise full rank pairing that the differentials of the shear coordinates in the subset are pointwise linearly independent on Teichm\"{u}ller space.  It also follows that the shear coordinates from the subset give a basis for the vector space of balanced sums of length functions. 

In \cite[\S 4]{Wlsymp}, we found for surface fundamental group representations into $PSL(2;\mathbb R)$ that the Poisson bracket of trace functions is a sum of trace functions.  The present result describes a simpler structure.  By construction Thurston shears on a common set of ideal geodesics commute and accordingly the Poisson bracket of Hamiltonian potential length functions is constant.

We now express the $2$-form $\omega$ in terms of $h$-lengths and use the formula to give the relation to Corollary \ref{hform1}.

\begin{corollary} \label{hh}For an ideal triangulation $\Delta$, the pullback WP K\"{a}hler form is
\[
\widetilde{\omega_{WP}}\,=\,\sum_{\operatorname{cusps}}\, \sum_{j=1}^p\, 
\widetilde{h}_{j}\wedge\widetilde{h}_{j+1},
\]
where the first sum is over cusps, the second sum is over $h$-lengths at a cusp enumerated in counterclockwise cyclic order and $\widetilde{h}_{*}=d\log h_*$. For an ideal triangulation $\Delta$, the pullback WP K\"{a}hler form is also given as
\[
\widetilde{\omega_{WP}}\,=\, \frac12\sum_{e\in\Delta}d\log \lambda_e\wedge d\sigma_e.
\]
\end{corollary}
\begin{proof} We begin with shear coordinates for $\caT$ and the shear pairing 
$\omega_{WP}(\sigma_{\mathcal A},\sigma_{\mathcal B})$ of Corollary \ref{commshear} above. The coefficients $\{a_j\},\{b_j\}$ are the evaluations of the differentials $\{d\sigma_e\}$ of the shear coordinates on the Thurston shears $\sigma_{\mathcal A},\sigma_{\mathcal B}$. From (\ref{shear}) and Figure \ref{fig:diamond}, the differential of a shear coordinate is $d\log h''/h'$ where $h''$ is the $h$-length clockwise from the edge and $h'$ is the $h$-length counterclockwise from the edge.  We now write the sum (\ref{2form}) at a cusp in terms of increments of $h$-lengths.  We use the notation of formula (\ref{2form}).  Introduce a decoration for the surface and write the shear coordinate increments in terms of $h$-length increments as $a_j=\widetilde{h}_{j-1}-\widetilde{h}_{j}$ and $b_j=\widetilde{g}_{j-1}-\widetilde{g}_{j}$, where $\widetilde{h}_{*}, \widetilde{g}_{*}$ are now the evaluations of the differential $d \log h_*$.  The partial sums are $A_0=0$ and $A_k=\sum_{j=1}^ka_j=\widetilde{h}_{p}-\widetilde{h}_{j}$, where with the cyclic ordering $\widetilde{h}_{0}=\widetilde{h}_{p}$ and by hypothesis $\sum_{j=1}^p\widetilde{h}_{j}=0$.  We find the contribution to $\omega$ from an individual increment $\widetilde{g}_{k}$ by considering
\begin{multline*}
(A_k+A_{k-1})b_k\,+\,(A_{k+1}+A_{k})b_{k+1}\,=\\
(2\widetilde{h}_{p}-\widetilde{h}_{k}-\widetilde{h}_{k-1})(\widetilde{g}_{k-1}-\widetilde{g}_{k})\,+\,(2\widetilde{h}_{p}-\widetilde{h}_{k+1}-\widetilde{h}_{k})(\widetilde{g}_{k}-\widetilde{g}_{k+1}).
\end{multline*}
The overall contribution is  
$(\widetilde{h}_{k-1}-\widetilde{h}_{k+1})\widetilde{g}_{k}$.  We now have that
\begin{multline*}
\omega\,=\,\frac12\sum_{k=1}^p(A_k+A_{k-1})b_k\,=\\ 
\frac12\sum_{k=1}^p\det
\begin{pmatrix} \widetilde{h}_{k-1} & \widetilde{h}_{k} \\ \widetilde{g}_{k-1} & \widetilde{g}_{k} \end{pmatrix}
 \,=\,\sum_{k=1}^pd\log h_{k-1}\wedge d\log h_k\,(\sigma_{\mathcal A},\sigma_{\mathcal B})
\end{multline*}
and the first formula is established.  

The second formula follows from Theorem \ref{Penlambda} and formal considerations.  From formula (\ref{shear}) we have that
\[
d\log\lambda_e\wedge d\sigma_e\,=\,\widetilde{\lambda}_a\wedge\widetilde{\lambda}_e+\widetilde{\lambda}_e\wedge\widetilde{\lambda}_b +
\widetilde{\lambda}_c\wedge\widetilde{\lambda}_e +
\widetilde{\lambda}_e\wedge\widetilde{\lambda}_d,
\]
where the ordered side pairs $(a,e),(e,b),(c,e)$ and $(e,d)$ are in counterclockwise order relative to their containing triangles.  The pairs are the side pairs of Figure 
\ref{fig:diamond} with one side a diagonal. Now given a pair of adjacent sides of the triangulation $\Delta$, the pair occurs in two quadrilaterals with one of the sides being a diagonal.  It follows that the sum of $d\log\lambda_e\wedge d\sigma_e$ over edges is twice the sum of Theorem \ref{Penlambda}.  The second formula follows.
\end{proof}

An observation of Joergen Andersen provides a direct relation of the above to Corollary \ref{hform1}.  The coupling equation $h_1h_2=h_3h_4$ gives the $2$-form equation $\widetilde h_1\wedge \widetilde h_2+\widetilde h_2\wedge \widetilde h_3+\widetilde h_3\wedge \widetilde h_4+\widetilde h_4\wedge \widetilde h_1=0$ for $\widetilde h_*=d\log h_*$. The relation  
$\widetilde h_1\wedge \widetilde h_2+\widetilde h_3\wedge \widetilde h_4\,=\,\widetilde h_3\wedge \widetilde h_2+\widetilde h_1\wedge \widetilde h_4$ follows. Beginning with Corollary \ref{hform1} and referring to Figure \ref{fig:diamond}, we observe the following.  For an edge $e$ of the triangulation, the wedge of $h$-lengths adjacent to $e$ of the triangles adjacent to $e$ can be replaced with the wedge of $h$-lengths for consecutive vertex sectors at the cusps at the ends of $e$.  The replacement agrees with the orientations of the formulas.  The replacement for each edge of the triangulation transforms the first adjacent by side formula to the second adjacent by vertex formula.   

\begin{example}\textup{The form $\omega$ for a once punctured torus.}\end{example}
\vspace{-.1in}
\noindent  A choice of three disjoint ideal geodesics decomposes a once punctured torus into two ideal triangles. The torus is described by edge identifying two ideal triangles to form a topological rectangle with diagonal $\gamma$, and then separately identifying the horizontal edges $\alpha$ and vertical edges $\beta$.  The pattern of geodesics at the cusp is twofold $\alpha,\gamma,\beta$.  Consider the triples of balanced weights $\{a,b,-a-b\}$ and $\{c,d,-c-d\}$ for the sequence $\alpha,\beta$ and $\gamma$. For the geodesics enumerated according to the pattern at the cusp, the sequence of partial sums for the second set of weights is $A_0=0,A_1=c,A_2=-d$ and $A_3=0$.  The sum (\ref{2form}) evaluates to $(ca+(c-d)(-a-b)+-db)=(ad-bc)$.

\vspace{.12 in}

We now follow the discussion of Bonahon \cite[Theorem 15]{BonTran} and Harer-Penner \cite[Section 2.1]{HP} for the dimension of the space of balanced sum coefficients. 

\begin{lemma}\label{edgebasis}
For a surface with cusps and a maximal configuration of disjoint ideal geodesics, the space of balanced sum coefficients has the same dimension as the Teichm\"{u}ller space.
\end{lemma}
\begin{proof}  Consider a configuration of ideal geodesics with weights as a graph with weighted edges.  The graph is connected since ideal triangles fill in the configuration to form a connected surface. We will sequentially coalesce and remove edges, each time decreasing the number of vertices, to finally obtain a single vertex graph.  For a surface with a single cusp no coalescing of edges is necessary.  Otherwise by connectedness, there is an ideal triangle with not all vertices at the same cusp.  
Begin with such a designated triangle.  If only two vertices are at distinct cusps, then we begin by coalescing an edge connecting the distinct vertices.  
If all vertices are at distinct cusps then we begin by sequentially coalescing two edges of the triangle and the third edge will not be subsequently coalesced.  We label the ends of edges as {\em incoming} or {\em outgoing} at coalesced vertices as follows.  Label the ends of edges adjoining the first vertex as {\em incoming}.  Coalesce the first designated edge, remove the weight and label the remaining ends of edges at the second vertex as {\em outgoing} for the coalesced vertex.  At the coalesced vertex the weight condition is that the sum of incoming weights equals the sum of outgoing weights.  
To continue, take a path of edges to an uncoalesced vertex and coalesce the first edge to an uncoalesced vertex along the path. 
Label the new ends of edges at the coalesced vertex as the opposite type as for the initial segment of the coalesced edge.  At the coalesced vertex the weight condition continues to be that the sum of incoming weights equals the sum of outgoing weights.  Continue coalescing edges until only a single vertex remains.  For a surface of genus $g$ with $n$ cusps, there are $6g-6+3n$ edges in a maximal configuration.  A total of $n-1$ edges are coalesced and $6g-6+2n+1$ edges remain. At least one edge of the initial designated triangle gives rise to an incoming-incoming edge of the final coalesced vertex.  The single weight sum relation is a non trivial condition for the weight on the incoming-incoming edge.  The space of weights on the final graph has the expected dimension.  
\end{proof}    


\section{The Fock shear coordinate algebra}\label{alge}

Fock and Goncharov in their quantization of Teichm\"{u}ller space introduced and worked with a Poisson algebra for the shear coordinate functions \cite{Fk,FkChk,FkGn}.  The quantization considerations begin with the Fock-Thurston Theorem that for any ideal triangulation, the corresponding shear coordinates (without the vanishing sums about cusps condition) provide a real-analytic homeomorphism of the holed Teichm\"{u}ller space to Euclidean space \cite[Chap. 4, Theorem 4.4]{Penbk}.  Fock proposed a Poisson structure by introducing a natural bivector, an exterior contravariant $2$-tensor $\eta$ and defining $\{f,g\}=\langle (df, dg),\eta\rangle$ for $f,g$ smooth functions.  A relationship to the WP K\"{a}hler form was also proposed.  A bivector defines a Poisson structure with Jacobi identity provided its Schouten-Nijenhuis tensor vanishes. 

\begin{theorem}\textup{\cite{Fk,FkChk}}  
For an ideal triangulation $\Delta$ and corresponding shear coordinates, the bivector
\[
\eta_{\Delta}\,=\,\sum_{\Delta} \frac{\partial\ }{\partial\sigma_{a}}\wedge\frac{\partial\ }{\partial\sigma_{b}}+\frac{\partial\ }{\partial\sigma_{b}}\wedge\frac{\partial\ }{\partial\sigma_{c}}+\frac{\partial\ }{\partial\sigma_{c}}\wedge\frac{\partial\ }{\partial\sigma_{a}}
\]
is natural for the holed Teichm\"{u}ller space, where the individual triangles have sides $a,b$ and $c$ in counterclockwise order.
\end{theorem} 

Penner gave a topological description of the bracket of shear coordinates
 \cite[pg. 81]{Penbk}, a proof that the bivector is independent of triangulation and also determined the center of the algebra \cite[Chap. 2]{Penbk}.  For the topological description of the bracket, recall the definition of the {\em fat graph} dual to an ideal triangulation.  To construct the fat graph $G$ embedded in the surface, choose a vertex interior to each triangle and connect vertices by an edge when triangles are adjacent.  The result is a trivalent graph with a cyclic ordering of edges at each vertex.  The trivalent graph is a deformation retract of the surface.  

Penner's topological description of the bracket is the following \cite[pg. 81]{Penbk}.  Consider an ideal triangulation $\Delta$ with dual fat graph spine $G$.  If $a,b\in\Delta$ are distinct edges, then let $\epsilon_{ab}$ be the number of components of the complement of $\Delta\cup G$ whose frontier contains points of $a$ and $b$, counted with a positive sign if $a$ and $b$ are consecutive in the counterclockwise order in the corresponding region, and with a negative sign if $a$ and $b$ are consecutive in the clockwise order.\footnote{We have reversed Penner's original sign convention given that his bivector has sides enumerated in a clockwise order, while Fock's bivector has sides enumerated in a counterclockwise order.}  Setting $\epsilon_{aa}=0$ for each $a\in\Delta$, $\epsilon_{ab}$ takes the possible values $0,\pm 1,\pm2$ and comprises a skew-symmetric matrix indexed by $\Delta$.  The quantity $\epsilon_{ab}$ is the count of oriented vertex sectors jointly bounded by $a$ and $b$. 

\begin{definition} The Fock shear coordinate algebra is defined by the bracket 
$\{\sigma_a,\sigma_b\}\,=\,\epsilon_{ab}$ for $a,b\in\Delta$.
\end{definition} 

From formula (\ref{shear}) and Figure \ref{fig:diamond}, a shear coordinate is a balanced sum of length functions.  For Riemann surfaces with cusps the WP Poisson bracket of sums of length functions is given in Corollary \ref{commshear} in terms of weights and the form $\omega$.  We evaluate $\omega$ for quadrilaterals and find that the evaluation agrees with Penner's topological description of the count $\epsilon_{ab}$.  

\begin{theorem}\label{FkWP}  The Fock shear coordinate algebra is the WP Poisson algebra.  The Fock shear coordinate bracket is given by the form $\omega$.
\end{theorem}

\begin{proof}
We begin with Corollary \ref{commshear} providing that the Poisson bracket of the shear coordinates for edges $e,f$ is $\{\sigma_e,\sigma_f\}=2\sum_{\operatorname{cusps}}\omega(\{a_j\},\{b_j\})$, for $\{a_j\},\,\{b_j\}$ the weights for the shears as sums of lengths of ideal geodesics.  The matter is to evaluate the sum (\ref{2form}) for $\omega$ for the possible configurations.  We first consider the case of the quadrilateral for the side $e$ embedded in the surface and then describe necessary modifications for sides of the quadrilateral coinciding.  The quadrilateral with weights for the edge $e$ is given in Figure \ref{fig:diamondwt}.

\begin{figure}[htb] 
  \centering
  \includegraphics[bb=0 0 551 573,width=3in,height=3.12in,keepaspectratio]{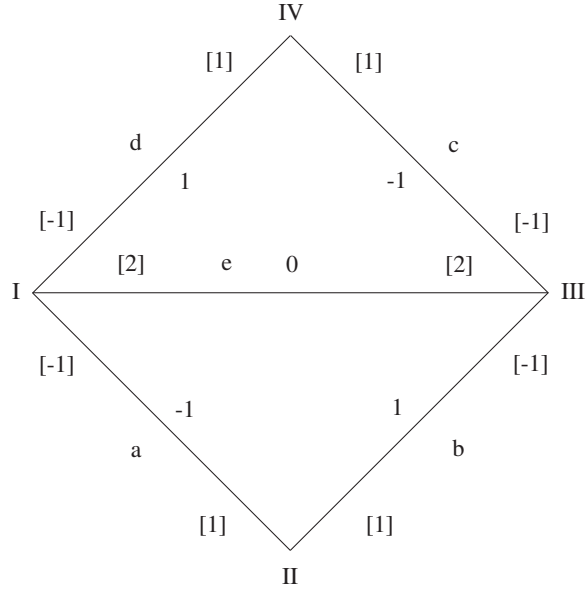}
  \caption{The quadrilateral for a triangulation edge $e$ following formula (\ref{shear}). The quadrilateral sides are labeled by lower case letters and vertices are labeled by Roman numerals.  The edge weights $0,\pm1$ refer to expressing the $e$ shear coordinate as a balanced sum of edge lengths.  The numbers in square brackets are the sums $A_j+A_{j-1}$.}
  \label{fig:diamondwt}
\end{figure}
  
Referring to formula (\ref{2form}), the first calculation is for the partial sums $A_j$ of edge weights.  At a vertex, edges are enumerated for summation in the counterclockwise order with the {\em first} edge being the clockwise most edge.  Normalize the partial sums to be zero for the not listed edges preceding the first edge.  The partial sums by vertex and in counterclockwise order are given in Table \ref{parsum}.  The second calculation is for the sums $A_j+A_{j-1}$ of partial sums about vertices.  The sums are given in Figure \ref{fig:diamondwt} by the numbers in square brackets; again sums vanish for edges not listed.  Now we are ready to consider the configuration of the quadrilateral for the edge $f$ and the sum of weights $\frac12(A_j+A_{j-1})b_j$.  The weights for $f$ are again $0,\pm1$ as in Figure \ref{fig:diamondwt}.  The edges $e$ and $f$ are necessarily distinct.  First consider that $f$ coincides with a boundary edge of the $e$ quadrilateral.  In this case the diagonal edge weight $0$ for $f$ is multiplied by the $[\pm1]$ boundary edge weights for $e$ and added to the $\pm1$ boundary edge $f$ weight times $1/2$ the sum of the $[2]$ and $[2]$ diagonal weights for $e$.  The result is $\pm2$ with the positive sign if $f$ is counterclockwise from $e$.  Now consider the case that the $e$ and $f$ quadrilaterals are either disjoint or intersect along a boundary edge. In the case of intersection along a boundary, the vanishing sum $[1]+[-1]$ of $e$ boundary weights gives a vanishing overall contribution.   This completes the calculation if the quadrilateral of $e$ is embedded. 

In general a pair of sides of the quadrilateral of $e$ could coincide; we do not consider the special cases $(g,n)=(0,3)$ or $(1,1)$ where two side pairs coincide.  A pair of adjacent sides could coincide by a $3/4$ rotation about the common vertex or opposite sides could coincide by a translation.  When sides coincide the contribution to $\omega$ is found by adding the contributions from each of the relative configurations for the quadrilateral of $f$.  The result will be $0,\pm4$ according to adjacent or opposite sides coinciding and the $e,f$ orientation.  As already noted, we are using the adjustment \cite[\S 5]{Wlcusps} to our formulas $2\omega_{WP}(\ ,t_*)=d\ell_*$ in place of $\omega_{WP}(\ ,t_*)=d\ell_*$ systematically used by Penner and Fock.  The consequence is that our shear pairing is fourfold the Fock and Penner calculations. With this information, the shear pairing evaluations correspond and the proof is complete.  

\begin{table}
\begin{center}
\renewcommand{\arraystretch}{1.3}
\begin{tabular}{|c|c|c|c|}
\hline
Vertex & $A_1$ &  $A_2$ & $A_3$ \\ \hline
I  & -1  & -1  & 0 \\ \hline
II  & 1  & 0 & \\ \hline
III  & -1  & -1  & 0 \\ \hline
IV  & 1  & 0 & \\ \hline
\end{tabular}
\caption{Partial weight sums in counterclockwise order about vertices.}
\label{parsum}
\end{center}
\end{table}

\end{proof}

\section{The norm of a length gradient for a collar crossing geodesic}\label{riemmgeom}

We continue to consider compact surfaces with crossing geodesics $\alpha$ and a reflection symmetry, see Figure \ref{fig:crossing}. We consider surfaces $R_{\epsilon}$ obtained by doubling a surface with cusps with ideal geodesics $\alpha$, and opening cusps to obtain short length core geodesics $\gamma$.  We are interested in the products of the gradients $\grad\lla$ and $\grad\llg$.  Theorem \ref{gradpr} and Lemma \ref{bglfb} can be combined to provide expansions for the pairings
\[
\langle\grad\llg,\grad\llg\rangle\,=\,\frac{2}{\pi}\llg\,+\,O(\llg^4)
\]
and
\[
\langle\grad\lla,\grad\llg\rangle\,=\,\frac{-4}{\pi}(\#\alpha\cap\gamma)\,+\,O(\llg^2).
\]
Considerations of Chatauby convergence and sums of the differential $\Omega^2$ from Section 2 suggest the heuristic expansion $\grad\ell_{\alpha}=c_{\alpha}(\ell_{\gamma})\grad\ell_{\gamma}+\overline{\psi(\ell_{\gamma})}(ds^2)^{-1}$ with $\psi(\ell_{\gamma})\in Q(R_{\epsilon})$ converging to $\psi(0)\in Q(R\cup\bar R)$.  A simple argument provides that $\psi(0)$ is orthogonal to the limit of $\grad\ell_{\gamma}$. The above pairing formulas and heuristic then suggest an expansion

\[
\langle\grad\lla,\grad\lla\rangle\,=\,\frac{8}{\pi\llg}(\#\alpha\cap\gamma)^2\,+\,O(1).
\]
The divergence of the pairing corresponds to the geometry.  The limit of $d\lla$ is formally the differential of length of an ideal geodesic and is a holomorphic quadratic differential with double poles at cusps.  The limit is not an element of $Q(R)$.  Also the limiting infinitesimal deformation $\grad\ell_{\alpha}$ corresponds to opening cusps and has infinite WP norm.

We would like to now use the gradient pairing formula, Theorem \ref{gradpr}, to find the WP pairing for balanced sums of gradients of lengths of ideal geodesics.  The above considerations show that a pairing formula involves canceling divergences in $\llg$. The divergences appear directly in evaluating the formula.  The crossing geodesic $\alpha$ is orthogonal to the collar core $\gamma$.  Arcs along $\gamma$ connect the intersection points with $\alpha$.  Each connecting arc provides a summand for the Theorem \ref{gradpr} evaluation.  The connecting arcs along $\gamma$ occur in families; a family consists of a simple arc and the additional arcs obtained by adjoining complete circuits of $\gamma$.  With $\llg$ tending to zero and the summand $R(\cosh \mbox{dist})\approx2\log2/\mbox{dist}$ for small distance, there is an immediate divergence.   We consider the sequence of lengths as a partition for a Riemann sum and find the $\llg$-asymptotics of the sum. 

The resulting formulas involve an elementary function, a reduced length for an ideal geodesic and a reduced connecting arcs sum formula.  
\begin{definition}
For $0\le a\le 1$, define the function $\lambda(a)=a(1-a)/(2\sin\pi a)$ with value given by continuity at the interval endpoints.  For a crossing geodesic $\alpha$ on a compact surface $R$ with reflection symmetry, the reduced length $\red(\lla)$ is the signed length of the segment connecting length $1$ boundaries of the complement of collars about core geodesics.   For an ideal geodesic $\alpha$ on a surface with cusps, the reduced length $\red(\lla)$ is the signed length of the segment of $\alpha$ connecting the length $1$ horocycles about the limiting cusps. 
\end{definition}

The function $\lambda(a)$ is symmetric about $a=1/2$ and satisfies $1/8\le\lambda\le 1/2\pi$.  For a pair of points $p,q$ on a circle, we write $\lambda(p,q)$ for the evaluation using the fractional part of the segment from $p$ to $q$.  For a hyperbolic surface without cone points the length $1$ horocycles are embedded circles bounding disjoint cusp regions and $\red(\lla)$ is non negative.  For surfaces with cone points, the reduced length can be negative.  

For crossing geodesics $\alpha,\beta$ on a surface with reflection symmetry or ideal geodesics $\alpha,\beta$ on a surface with cusps, we will write 
\[
{\sum}^{red}_{\alpha\operatorname{ to }\beta}\mathcal R
\]
for the reduced sum over homotopy classes rel the closed sets $\alpha,\beta$ of arcs connecting $\alpha$ to $\beta$, that are not homotopic to arcs along a core $\gamma$ or along a horocycle.  For the double of a surface with cusps, the symmetric homotopy classes are even with respect to the reflection; for this situation the sum is only over arcs with representatives on a chosen side of the surface.  Each geodesic representative for the reduced sum intersects the thick subset of the surface and the reduced sum includes any intersection points of the ideal geodesics $\alpha$ and $\beta$.  We assume the main result Theorem \ref{shpr} and illustrate the approach with the example of a single core geodesic.  The general formula depends on the pattern of crossing geodesics. 
\begin{example}\textup{Expansion of the WP gradient pairing for crossing geodesics $\alpha,\beta$ and a single core geodesic $\gamma$.}  For the core intersections $\alpha\cap\gamma=\{a_1,a_2\},\, \beta\cap\gamma=\{b_1,b_2\}$ and a given positive constant $c$ then
\begin{multline*}
\langle\grad\lla,\grad\lla\rangle\,= \\ \,\frac{2}{\pi}\big(\frac{16}{\llg}\,+\,\red(\lla)\,+\,4\,+\,2\sum_{(a_i,a_j)}\log\lambda(a_i,a_j)\big)\,+\, 2\,{\sum}^{red}_{\alpha\operatorname{ to }\alpha}\mathcal R\,+\,O(\llg^{1-c})
\end{multline*}
and for $\alpha\ne\beta$
\begin{multline*}
\langle\grad\lla,\grad\llb\rangle\,= \\ \,\frac{2}{\pi}\big(\frac{16}{\llg}\,+\,2\sum_{(a_i,b_j)}\log\lambda(a_i,b_j)\big)\,+\, 2\,{\sum}^{red}_{\alpha\operatorname{ to }\beta}\mathcal R\,+\,O(\llg^{1-c}).
\end{multline*}
\end{example}

We are ready to consider that pairings of balanced sums on a surface with cusps are the limits of pairings of balanced sums on approximating symmetric compact surfaces. 
The balanced sum condition will serve to cancel the universal $16/\llg$ leading divergence terms.  To compare formulas note that a surface with cusps represents half of a compact surface.  It is also important that remainder terms as in the example tend to zero with  $\llg$ .  

We state the main result.  For a surface with cusps, the sum over core geodesic intersections is replaced with a double sum.  First, a sum over cusps and second, a sum over ordered pairs of ideal geodesic segments limiting to a cusp. Ideal geodesics are orthogonal to horocycles.   The fractional part of a horocycle defined by a pair of ideal geodesics is independent of the choice of horocycle.  The geometric invariant $\lambda$ is evaluated by considering the intersections with any horocycle for the cusp.  We present the formula for the case of a torsion-free cofinite group.  

\begin{theorem}\label{shpr} \textup{The ideal geodesic complex gradient pairing.} For a surface $R$ with cusps and balanced sums $\mathcal A=\sum \mathfrak a_j\ell_{\alpha_j}, \mathcal B=\sum\mathfrak b_k\ell_{\beta_k}$ of ideal geodesic length functions, the WP pairing of gradients is
\begin{multline*}
\langle\grad L(\mathcal A),\grad L(\mathcal B)\rangle\,=\,\\
\sum_{j,k}\mathfrak a_j\mathfrak b_k\bigg(\delta_{\alpha_j\beta_k} \frac{2}{\pi}(\red(\ell_{\alpha_j})+2)\,+  
\,\frac{2}{\pi}\sum_{\operatorname{cusps}}\,
\sum_{\begin{smallmatrix}\operatorname{segments\,}\tilde\alpha_j,\tilde\beta_k \\ \operatorname{ limiting\,to\,the\,cusp}\end{smallmatrix}}\log\lambda(\tilde\alpha_j,\tilde\beta_k) \\
\,+\,{\sum}^{red}_{\alpha_j\operatorname{ to }\beta_k}\mathcal R \bigg).
\end{multline*}
The first sum is over weights; the double sum is over ordered pairs of geodesic segments limiting to cusps.  The final sum is over homotopy classes rel the closed sets $\alpha_j,\beta_k$ of arcs connecting $\alpha_j$ to $\beta_k$, arcs that are not homotopic into a cusp.  For the homotopy class of an intersection $\alpha_j\cap\beta_k$, the function $\mathcal R$ is evaluated on $\cos \theta$, $\theta$ the intersection angle.  Otherwise, the function $\mathcal R$ is evaluated on the hyperbolic cosine of the length of the unique minimal connecting geodesic segment.
Twist-length duality and $J$ an isometry provide that 
$4\langle \sigma_{\mathcal A},\sigma_{\mathcal B}\rangle\,=\,\langle\grad\mathcal A,\grad\mathcal B\rangle$.  
\end{theorem}
\noindent\emph{Proof.}  Begin the consideration with compact surfaces with reflection symmetries and balanced sums of geodesic-length functions converging to a surface with cusps formally doubled across the cusps.  The approach is to show that the connecting arcs sums of Theorem \ref{gradpr} converge to the sum for the limiting surface.  The individual summands are considered in terms of the geometry of the biorthogonal connecting geodesic segments.  

Begin by normalizing the uniformizations to ensure Chabauty convergence of the deck transformation groups $\Gamma$.   For the geodesic $\alpha$, let $\tilde\alpha$ be a chosen geodesic line lift and $\langle A\rangle$ the cyclic group stabilizer.  A fundamental interval on $\tilde\alpha$ is chosen; each left $\langle A\rangle$ orbit in $\Gamma\tilde\alpha$ and $\Gamma\tilde\beta$, $\tilde\beta$ a lift of $\beta$, has a unique biorthogonal geodesic connecting segment with one endpoint in the $\tilde\alpha$ fundamental interval.  The considerations proceed in terms of the geometry of the second endpoint of the connecting segment.  The finite number of terms corresponding to endpoints in a given compact set converge.  The sums for families of connecting segments along the core geodesics provide universal divergences; the analysis is described in the next section.   
The remaining connecting segments have second endpoint outside a given compact set and the segments do not lie along core geodesics.  The remaining segments necessarily intersect the lift of the thick subset.   The remaining segments are treated according to whether the second endpoint lies in the lift of the thick or the thin subset.  In the first case, the injectivity radius is bounded away from zero and the sum of such terms is uniformly bounded by applying the distant-sum method of \cite[Chap. 8]{Wlcbms}.  In the second case, the endpoint lies in the lift of a standard collar or cusp region.  Hyperbolic geometry is used to show that the full sum over the stabilizing cyclic hyperbolic or parabolic group is bounded simply by the distance of the fundamental interval on $\tilde\alpha$ to the boundary of the region.  The distant-sum and cyclic group bounds provide that the contributions from the complement of a large compact set is sufficiently small.  The estimates for the various cases are combined to establish convergence of formulas.  

We consider the connecting segments along a given core geodesic.  We outline the approach and give a detailed treatment in the next section.  The sum for a family of connecting arcs in a given direction along a core geodesic has the form 
\[
\sum^{\infty}_{n=0}S((a+n)\ell)\quad\mbox{for}\quad S(t)=\cosh t\,\bigg(\log\frac{\cosh t+1}{\cosh t-1}\bigg)-2
\]
for $\ell$ the core length and $a\ell, a>0,$ the distance between core intersection points.  The function $S(t)$ has the initial expansion $S(t)\approx 2\log 2/t$ and for $N$ approximately $\ell^{-1-\epsilon}, \epsilon>0,$ we break up the sum
\begin{multline*}
\sum^{\infty}_{n=0}S((a+n)\ell)\,=\\  \sum^N_{n=0}2\log\frac{2}{(a+n)\ell}\,+\,\frac{1}{\ell}\sum^N_{n=0}\ell\bigg(S((a+n)\ell)-2\log\frac{2}{(a+n)\ell}\bigg)\,+\,\sum^{\infty}_{n=N+1}S((a+n)\ell).
\end{multline*}
For the first sum, we use additivity of the logarithm to obtain an expression in terms of $\log 2/\ell$ and $\log\Gamma(a+1)$ for the gamma function.  Stirling's formula is then applied.  For the second sum, half of the first and last sum terms are separated, then the Trapezoid Rule is applied to approximate the sum by an integral and an error term.  The Trapezoid Rule provides an improved approximation in $\ell$.  The integral is calculated by an antiderivative.  Finally the bound that $S(t)$ is $O(e^{-2t})$ for $t\ge t_0 >0$, provides that the third sum is exponentially small; the consequence is that for $a>0$ the original full sum has the expansion
\[
\frac{2}{\ell}\ +\ \log\frac{\Gamma(a+1)^2\ell^{2a-1}}{2^{2a}\pi}\ +\ 2a-1\ +\ O(\ell^{1-\epsilon}).
\]
The overall expansion for connecting arcs in the forward and reverse directions is obtained by combining the expansions for the values $a$ and $1-a$.  Identities for the gamma function are used to simplify the resulting formula and to obtain the function $\lambda$.  As already noted, the  $\ell$-divergence is in the leading term.  The balanced sum condition provides for the overall canceling of divergences in evaluating the gradient product. The proof  is complete. $\qedd$

\begin{corollary}
For a balanced sum $\mathcal A=\sum \mathfrak a_j\ell_{\alpha_j}$ of ideal geodesic length functions and $\beta$ a closed geodesic, the shear and twist derivative pairing is
\[
\sigma_{\mathcal A}\llb\,=\,-t_{\beta}L(\mathcal A)\,=\,\sum_j\mathfrak a_j\sum_{p\in\alpha_j\cap\beta}\cos\theta_p
\]
for the intersection angles measured from $\alpha_j$ to $\beta$.
\end{corollary}    
   
%

\begin{example}\label{Dedekind} \textup{A distance relation for the elliptic modular tessellation.}
\end{example}
\begin{figure}[htbp] 
  \centering
  \includegraphics[bb=0 0 624 266,width=5.1in,height=2.17in,keepaspectratio]{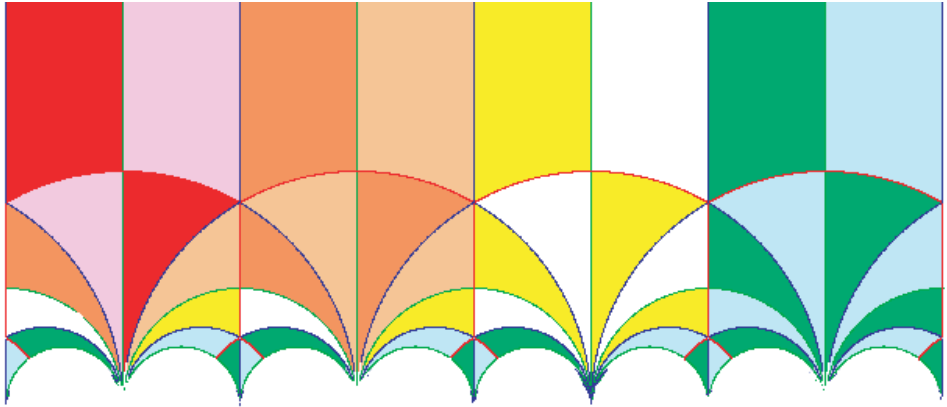}
  \caption{The Dedekind tessellation.  Graphic created by and used with permission from Gerard Westendorp.}
  \label{fig:modtess}
\end{figure}

\vspace{-.1in}
The Dedekind tessellation is the tiling of the upper half plane for the action of 
$PSL(2;\mathZ)$.  The light, respectively dark, triangle tiles form a single $PSL(2;\mathZ)$ orbit.  The reflection in the imaginary axis normalizes the group and interchanges the light and dark triangles.  The tessellation vertices are fixed points of elements of the group action.  There are two orbits for vertices.  There are also two orbits for ideal lines.   The first consists of the lines containing a single order-$2$ fixed point.  The second consists of the lines sequentially containing an order-$3$, an order-$2$ and an order-$3$ fixed point.  We refer to the types as $2$-lines and $323$-lines.  We consider the lines with weights: $w=+1$ for $323$-lines and $w=-1$ for $2$-lines.  The system of weighted lines is $PSL(2;\mathZ)$ invariant. 

The formula of Theorem \ref{shpr} provides a relation for the distances between lines for the Dedekind tessellation.  For any choice $\tilde a$ of a $323$-line and $\tilde \alpha$ of a $2$-line, we have
\[ 
\sum_{\operatorname{ultraparallels\ to\,}\tilde a}w(\eta) R(d(\tilde a,\eta))\ -  
\sum_{\operatorname{ultraparallels\ to\,}\tilde \alpha}w(\eta) R(d(\tilde \alpha,\eta)) \ =\ \log\frac{3^6\pi^4}{2^{26}}
\]
for $R(d)=u\log ((u+1)/(u-1))-2$ and $u=\cosh d$.  Ultraparallels are the tessellation lines at positive distance.  Lines at zero distance are asymptotic.  

We find the relation as an exercise in evaluating the formula of Theorem \ref{shpr}.   We begin with the geometry of the tiling quotient.   We work with the thrice-punctured sphere uniformized by the projectivized index $6$ subgroup $P\Gamma(2)\subset PSL(2;\mathZ)$ of matrices congruent to the identity modulo $2$.  A fundamental domain for the torsion-free group $P\Gamma(2)$ is given by the twelve light and dark triangles adjacent to a given largest height non vertical $323$-line.  The $P\Gamma(2)$ quotient is a tri-corner pillow with three $323$-lines, labeled $a,b,c$ and three $2$-lines, labeled $\alpha,\beta,\gamma$.  The $2$-lines separate the quotient into two ideal triangles.  A $323$-line enters a single cusp of the quotient, while a $2$-line connects two distinct cusps.  We evaluate the pairing product for the weighted balanced sum $\sigma=a+b+c-\alpha-\beta-\gamma$.  The sum is $P\Gamma(2)$ invariant, thus $\grad \sigma\in Q(P\Gamma(2))$ by Theorem \ref{twlth}.  The space of $P\Gamma(2)$ quadratic differentials is zero dimensional. The self pairing of $\grad \sigma$ is zero.

We determine the contributions for terms on the right hand side of the Theorem \ref{shpr} formula.  The evaluation corresponds to the formal expansion of the product $(a+b+c-\alpha-\beta-\gamma)^2$.  The pairing is real and the initial factor $\pi/2$ can be moved to the left hand side.  We begin with the reduced length contribution.  The $P\Gamma(2)$ cusps have width $2$; the length $1$ horocycle at infinity has height $2$.  
For a vertical $2$-line, half of the reduced length segment connects the height two horocycle to the order-$2$ fixed point at height $1$.  A $2$-line has reduced length 
$2\log 2$.  
For a vertical $323$-line, half of the reduced length segment connects the height two horocycle to the order-$2$ fixed point at height $1/2$.  
A $323$-line has reduced length $4\log 2$.  
The reduced length contributing terms of the product are $a^2+b^2+c^2+\alpha^2+\beta^2+\gamma^2$.  The total first term reduced length contribution is $18\log2 +12$.  We next consider the $\log \lambda$ contributions, which measure the geometry of the ideal geodesics limiting to cusps.  There are two reflections stabilizing each cusp.  The reflections stabilize the geodesics and provide that the intersections of the ideal geodesics with a horocycle are equally spaced and alternate by weights.  The $\log\lambda$ contributing terms of the product are
\begin{multline*}
a^2\,+\,b^2\,+\,c^2\,+\,\alpha^2\,+\,\beta^2\,+\,\gamma^2\, -\,2a\beta\,-\,2a\gamma\\ -\,2b\alpha\,-\,2b\gamma\,
-\,2c\alpha\,-\,2c\beta\,+\,2\alpha\beta\,+\,2\alpha\gamma\,+\,2\beta\gamma.
\end{multline*}
By $PSL(2;\mathZ)$ symmetry, the evaluation is the same as for $3a^2+3\alpha^2-12a\beta+6\alpha\beta$. 
The $a^2$ contribution is $2\log(\lambda(0)\lambda(1/2))$ given the two segments at a cusp; the $\alpha^2$ contribution is $2\log\lambda(0)$ given the two limiting cusps; the $a\beta$ contribution is $2\log\lambda(1/4)$ given the symmetry of $\lambda$ and the $\alpha\beta$ contribution is $\log\lambda(1/2)$.  The evaluations are $\lambda(0)=1/(2\pi)$, $\lambda(1/4)=3\sqrt2/32$ and $\lambda(1/2)=1/8$.  
The total $\log \lambda$ contribution is
\[
6\, \log\frac{1}{16\pi}\ +\ 6\, \log\frac{1}{2\pi}\ +\ -24\, \log\frac{3\sqrt2}{32}\,+\,6\,\log\frac18.
\]
We next consider the contribution from ideal geodesics intersecting. The intersection product contributing terms are $2ab+2ac+2bc-2a\alpha-2b\beta-2c\gamma$. The geodesic intersections $ab$, $ac$ and $bc$ are twofold.  From the formula the total intersection contribution is
\[
  2\cdot3\cdot  R(\cos\frac{\pi}{3})\ +\ 2\cdot3\cdot R(\cos\frac{2\pi}{3})\ -\ 2\cdot 3\cdot  R(\cos\frac{\pi}{2})=\ 6\log3\,-\, 12
\]
as follows.  The leading $2$-factors are from the formal expansion of $\sigma^2$.  The $3$-factors are from the symmetry of the triples $a,b,c$ and $\alpha,\beta,\gamma$.  The first and second terms correspond to the fact that distinct $323$-lines intersect twice.  The $R$-evaluations $R(\cos\pi/3)=(\log 3)/2-2$ and $R(\cos\pi/2)=-2$ are elementary.  The final contribution of the right hand side of the overall formula is the sum for the nontrivial connecting geodesics.   We start with the formal expansion $\sigma^2=a\sigma+b\sigma+c\sigma-\alpha\sigma-\beta\sigma -\gamma\sigma$.  By $PSL(2;\mathZ)$ symmetry the evaluation is the same as for $3a\sigma-3\alpha\sigma$.  Connecting geodesics are enumerated by lifting to the universal cover.  Given lifts $\tilde a$ and $\tilde\alpha$, the desired sums are obtained.   The overall relation now follows.  We note that the lines asymptotic to $\tilde a$ and $\tilde\alpha$ correspond to the limits of lines with connecting segments along core geodesics; the $\log \lambda$ terms account for the combined contribution of the asymptotic lines.  

\section{The geodesic circuit sum}\label{circ}
We consider the contribution to the Theorem \ref{gradpr} sum corresponding to connecting geodesics given by circuits about a fixed closed geodesic.  Such a circuit sum enters when the geodesics $\alpha$ and $\beta$ are orthogonal to a common closed geodesic.  The  summands are evaluations of the function
\[
S(t)\,=\,\cosh t\bigg(\log\frac{\cosh t +1}{\cosh t -1}\bigg)\,-\,2.
\]
The consideration is for the length parameter $\ell$ expansion of the infinite sum of circuits.  The application to Theorem \ref{shpr} requires an expansion with remainder term tending to zero for small $\ell$.  Simple analysis gives that the expansion begins with terms divergent in $\ell$.  We provide the expansion.
\begin{theorem}\label{lasymp} For $a$ and $\epsilon$ positive, the circuit sum has the expansion
\[
\sum_{n=0}^{\infty}S((a+n)\ell)\,=\,\frac{2}{\ell}\,+\,\log\frac{\Gamma(a+1)^2\ell^{2a-1}}{2^{2a}\pi}\,+\,2a-1\,+\,O(\ell^{1-\epsilon})
\]
for the gamma function $\Gamma(z)$.
\end{theorem}

\begin{corollary} For $\epsilon$ positive, the circuit sum for $0<a<1$ has the expansion
\begin{align}
\sum_{n=-\infty}^{\infty}S((a+n)\ell)\,&=\,\frac{4}{\ell}\,+\,2\log\frac{\Gamma(a+1)\Gamma(2-a)}{2\pi}\,+\,O(\ell^{1-\epsilon})\notag\\
&=\,\frac{4}{\ell}\,+\,2\log\frac{a(1-a)}{2\sin\pi a}\,+\,O(\ell^{1-\epsilon}),\notag
\end{align}
and for $a=1$ has the expansion
\[
\sum_{n=1}^{\infty}S(n\ell)\,=\,\frac{2}{\ell}\,+\,\log\frac{\ell}{4\pi}\,+\,1\,+\,O(\ell^{1-\epsilon}).
\]
\end{corollary}

\noindent\emph{Proof of Corollary.}  Since $S(t)$ is an even function the first sum can be rewritten as $\sum_{n=0}^{\infty}S((a+n)\ell)+S((1-a+n)\ell)$ and the theorem is applied.  The gamma function identities $\Gamma(z+1)=z\Gamma(z)$ and $\Gamma(1-z)\Gamma(z)\sin\pi z=\pi$ are applied to obtain the desired expression.  Finally the case $a=1$ is a direct application of the theorem.$\qedd$

\noindent\emph{Proof of Theorem.}  We begin with properties of the summand $S(t)$.  The summand has the small-$t$ expansion $S(t)=2\log2/t\,-\,2\,+\,O(t^2\log t)$ and the large-$t$ expansion $S(t)=O(e^{-2t})$.  We also consider the function
\[
F(t)\,=\,S(t)\,-\,2\log\frac{2}{t}
\]
and write
\[
F(t)\,=\,(\cosh t-1)\bigg(\log\frac{\cosh t+1}{\cosh t-1}\bigg)\ +\ \log\frac{t^2(\cosh t+1)}{4(\cosh t-1)}\ -\ 2.
\]
We note that for small-$t$, since $\cosh t-1$ is $O(t^2)$ and $t^2/(\cosh t-1)$ is analytic it follows that $F(t)$ has second derivative bounded by $-\log t$ for small-$t$.  

We are ready to begin the overall considerations and write the sum in the form of Riemann sums, adding in and subtracting out a $2\log 2/t$ contribution

\begin{align}\label{123}
\sum_{n=0}^{\infty}S((a+n)\ell)\,=\,&\sum_{n=0}^N2 \log\frac{2}{(a+n)\ell}\notag\\
&+\,\frac{1}{\ell}\sum_{n=0}^N\ell\big(S((a+n)\ell)-2\log\frac{2}{(a+n)\ell}\big)\notag\\&+\,\frac{1}{\ell}\sum_{n=N+1}^{\infty}\ell S((a+n)\ell)\notag\\
=\,&I\,+\,II\,+\,III.
\end{align}
We consider the right-hand sums in order.  For the first sum we have
\[
2\sum_{n=0}^N\log\frac{2}{(a+n)\ell}\,=\,2(N+1)\log\frac{2}{\ell}\,+\,2\sum_{n=0}^N\log\frac{1}{(a+n)}.
\]
The right hand sum is  $-2\log\prod_{n=0}^N(a+n)\,=\,-2\log \Gamma(a+N+1)/\Gamma(a+1)$.  We apply Stirling's formula $\log \Gamma(z)=\frac12\log2\pi/z\,+\,z(\log z\,-\,1)\,+\,O(1/z)$ to find that
\begin{multline*}
I\,=\,2(N+1)\log\frac{2}{\ell}\,+\,2\log\Gamma(a+1)\,-\,2(a+N+\frac12)\log(a+N+1)\\+\,2(a+N+1)\,-\,\log 2\pi\,+\,O(N^{-1})
\end{multline*}
and noting that $\log(a+N+1)=\log(a+N)+1/(a+N)+O(N^{-2})$ gives the desired final expansion
\begin{multline}\label{1}
I\,=\,2(N+1)\log\frac{2}{\ell}\,+\,2\log\Gamma(a+1)\,-\,2(a+N)\log(a+N)\\-\,\log(a+N+1)\,+\,2(a+N)\,-\,\log2\pi\,+\,O(N^{-1}).
\end{multline}

For the second sum of (\ref{123}) we use the Trapezoid Rule approximation for an integral.  The approximation involves weights $1/2$ for the first and last sum terms. The error bound is in terms of the second derivative of $F(t)$ on the interval $[a\ell,(a+N)\ell]$ and the square of the partition size.  The approximation gives the expansion
\begin{multline*}
II\,=\,\frac{1}{\ell}\int_{a\ell}^{(a+N)\ell}F(t)dt\,+\,\frac12\big(F(a\ell)\,+\,F((a+N)\ell)\big)\\+\,O(\ell\,|[a\ell,(a+N)\ell]|\max |F''|).
\end{multline*}
We set $(a+N)=\ell^{-\epsilon}$ and consider terms in order from right to left.  Given the small-$t$ logarithmic bound for $F''$ the remainder is bounded as $O(\ell^{1-2\epsilon})$.  Given the large-$t$ exponential decay $S(t)$ and the small-$t$ expansion of $S(t)$ then
\[
F((a+N)\ell)\,=\,-2\log \frac{2}{(a+N)\ell}\,+\,O(e^{-\ell^{-\epsilon}})\quad \mbox{and}\quad
F(a\ell)=-2+O(\ell^{2-\epsilon}).
\]
The next step is to include the contribution of sum $III$.  The sum is replaced with the corresponding integral.  Since the integrand is exponentially decreasing on the interval, the replacement remainder is exponentially small.  The considerations combine to give the expansion
\begin{multline*}
II\,+\,III\,=\,-\frac{2}{\ell}\int_{a\ell}^{(a+N)\ell}\log\frac{2}{t}\,dt\,+\,\frac{1}{\ell}\int_{a\ell}^{\infty}S(t)\,dt\\-\,1\,-\log\frac{2}{(a+N)\ell}\,+\,O(\ell^{1-2\epsilon}).
\end{multline*}
The first integrand has antiderivative $t\log 2/t\,+\,t$.  The second integrand $S(t)$ has antiderivative
\[
\sinh t\bigg(\log\frac{\cosh t+1}{\cosh t-1}\bigg),
\]
which has the large-$t$ expansion $2\,+\,O(e^{-2t})$.  We evaluate the integrals to find the contribution
\begin{multline*}
II\,+\,III\,=\,-2N\log 2\,+\,2(a+N)\log((a+N)\ell)\,-\,2(a+N)\\-2a\log a\ell\,+\,2a+\frac{2}{\ell}-2a\log\frac{2}{a\ell}\,-\,1\,-\log\frac{2}{(a+N)\ell}\,+\,O(\ell^{1-2\epsilon}).
\end{multline*}
The next step is to combine with expansion (\ref{1}) and note again that $(a+N)\ell=\ell^{-\epsilon}$ to find the desired final expansion
\[
I\,+\,II\,+\,III\,=\,\frac{2}{\ell}\,+\,\log\frac{\Gamma(a+1)^2\ell^{2a-1}}{2^{2a}\pi}\,+\,2a\,-\,1\,+\,O(\ell^{1-\epsilon}).
\]


\providecommand\WlName[1]{#1}\providecommand\WpName[1]{#1}\providecommand\Wl{W%
lf}\providecommand\Wp{Wlp}\def\cprime{$'$}

\end{document}